\documentclass[12pt]{amsart}

\let\cc\c
\usepackage{common}
\usepackage{mathrsfs}
\usepackage[lite]{amsrefs}

\title{On generically stable types in dependent theories}

\author{Alexander Usvyatsov}

\address{Alexander Usvyatsov\\
  University of California -- Los Angeles\\
  Mathematics Department\\
  Box 951555\\
  Los Angeles, CA 90095-1555\\
  USA}

\address{ Alexander Usvyatsov\\ Universidade de Lisboa \\
  Centro de Matem\'{a}tica e Aplica\cc{c}\~{o}es Fundamentais\\
  Av. Prof. Gama Pinto,2\\
  1649-003 Lisboa \\
  Portugal}

\urladdr{http://www.math.ucla.edu/\textasciitilde alexus}

\thanks{Research partially supported by FCT grant SFRH / BPD / 34893 /
2007}

\date{\today}

\begin{document}

\begin{abstract}
    We develop the theory of generically stable types, independence relation based on
    nonforking and stable weight in the context of dependent (NIP)
    theories.
\end{abstract}

\maketitle

\section{Introduction and Preliminaries}
\subsection{Introduction}
    The original motivation for this paper was generalizing certain aspects of the theory
    developed by Haskell, Hrushovski and Macpherson in \cite{HHM} for
    stably dominated types to a broader context. We believe that the right framework
    for most results (at least assuming the theory is dependent) has to do with
    ``stable'' types introduced by Shelah
    in \cite{Sh715}.
    Since the name ``stable'' had been used
    (e.g. by Lascar and Poizat, see \cite{LP}) for a different (much stronger)  notion before Shelah's paper
    was written, in order to avoid confusion, we use different
    terminology suggested by Hrushovski and Pillay and call our main
    object of study ``generically stable'' types.




    While the paper was being written, other particular cases of
    generically stable types became important for the study of theories interpretable in o-minimal
    structures carried out by Hasson, Onshuus, Peterzil and others.
    For example, notions of ``seriously stable'', ``hereditarily stable'' types
    were investigated in \cite{HaOn2}. Numerous conversations with
    Assaf Hasson and Alf Onshuus slightly changed the character of this work.

    We develop a cleaner and a more comprehensible
    theory of ``stable'' types than the one found in \cite{Sh715}. In particular we eliminate the need to work
    with finitely satisfiable types. This has
    two advantages: first, our approach allows one to avoid considering co-heir sequences
    (which we call Shelah sequences here) which used to create much confusion.
    Second, we provide a good picture of types over arbitrary sets,
    and not only over models or indiscernible sets.

    It is important to us, however, to show the connection between our and Shelah's
    approaches; therefore, sections 3 and 5 of the paper are devoted mostly to a
    systematic development of ``Shelah-stable'' types, giving a more complete picture than what
    is done in \cite{Sh715}.

    Together with deeper understanding came
    the realization that nonforking plays a central role in the general theory, and for
    generically stable
    types is equivalent to definability and gives rise to a nice independence relation, so we have
    a very smooth generalization of classical stability. In a sense this
    provides a complementary picture to the work of Dolich
    \cite{Dol} which characterizes forking in o-minimal theories
    (that is, ``as unstable as possible'' dependent theories).

    Some of our results can be found in a different form in a recent preprint by Hrushovski and
    Pillay \cite{HP} which was written simultaneously and independently
    of our work
    and is mostly focused on other issues such as invariant
    types and measures.

    Let us make a note on our choice of terminology. Following Lascar and
    Poizat, some people call a (partial) type $\pi(\x)$ stable if every extension of
    it is definable. If $\pi(\x)$ is a formula which defines in $\FC$ the
    set $D$, a more common terminology
    is ``$D$ is a stable stably embedded (definable) set''. If $T$ is dependent, then
    stability of a definable set has many equivalent definitions, as investigated
    e.g. by Onshuus and Peterzil in \cite{OnPe}. In particular, a set $D$ is stable
    if and only if it fails the order property if and only if
    it fails the strict order property, that is, there is
    no definable partial order with infinite chains on $D$.
    It follows that stable embeddedness comes ``for free'', that is, if $D$ is stable
    then every externally definable subset of $D$ is definable with
    parameters in $D$. So $D$ is stable
    if and only if the induced structure (with or without taking into account external parameters) on it is stable.
    Hence this terminology seems very reasonable to us and we will use it.

    It is also quite easy to see that Lascar-Poizat stability of a type $p$ is equivalent
    (assuming dependence) to the set of realizations of $p$ failing the order property
    (equivalently, the strict order property). This provides a justification for simply
    calling such types stable; still, we will restrain from doing so in order to avoid
    confusion between this and Shelah's terminology. So we'll call stable types in this strong
    sense ``Lascar-Poizat stable'' or ``hereditarily stable'' since in our
    context $p$ is Lascar-Poizat stable if and only if every extension of it is Lascar-Poizat stable
    if and only if every extension of it is generically stable.

    As for the term ``generically stable'', we think it captures the concept being studied here
    pretty well, since generically stable (that is, ``stable'' according to Shelah) types
    behave in a stable way ``generically'', i.e. when one takes nonforking extensions and Morley
    sequences.

    \bigskip
    The paper is organized as follows:

    Section 2 contains basic definitions and facts on indiscernible
    sequences, sets, splitting, forking, definability, etc in dependent
    theories. The main result of the section is Lemma
    \ref{lem:globalnonfork} which states (among other things) that
    the global average of a nonforking indiscernible set does not
    fork over the base set. This is easier if nonforking is replaced with
    nonsplitting, Lemma \ref{lem:seqstat}; for the nonforking case
    one needs to apply a more subtle analysis and understand the
    connections between different notions of splitting and forking
    in dependent theories.

    Section 3 is devoted
    to developing the basic theory of finitely satisfiable types, ultrafilters, Shelah sequences, etc. While it is
    essential for understanding Shelah's approach to generically stable types, it is not at all used in
    section 4 (where the main theory of generically stable types is developed), hence can be omitted in the
    first reading.

    Section 4 is the  central part of the article: we define generically stable types and prove most of their properties
    (such as definability and stationarity). We also show that
    generic stability is closed under parallelism.

    Section 5 is based on \cite{Sh715}, but we give a more complete
    and wide picture. In particular, we prove that when working over a slightly saturated model, being
    a generically stable type is equivalent to being both finitely satisfiable in and definable over a small subset.
    This result does not appear in \cite{Sh715}, and in fact was not known to Shelah at the time. We also
    give an example showing that for this criterion it is essential to consider type over saturated models.

    Section 6 presents a summary, connections to previous works on particular cases (stably dominated
    types \cite{HHM}, seriously stable types and hereditarily stable types
    introduced by Hasson and Onshuus in
    \cite{HaOn2}) and several examples of generically stable types that do not fall in any of the
    categories discussed above. These are also a good source of certain curious phenomena, showing the
    subtleties of working over sets as opposed to saturated models, differences between splitting
    and forking, which explain some of our earlier choices.
    In particular we see that nonsplitting extensions of
    generically stable types do not have to be generically stable, which
    can not happen with nonforking extensions.

    Section 7 is devoted to the original goal, developing the theory of independence for generically stable
    types, which happens to be quite easy once the general framework is well-understood. Independence
    relation for generically stable types turns out to be based on both nonforking and
    definability, generalizing classical stability. We also
    characterize generically stable types in terms of behavior of
    forking on the set of their realizations and show that a type
    stably dominated by a generically stable type is generically
    stable.

    Section 8 is the beginning of the theory of weight for generically stable types.
    We show that in a strongly dependent
    theory every generically stable type has finite weight. We also
    define stable weight of an arbitrary type, hoping that this
    will help us in understanding the ``stable part'' of a type in a
    dependent theory, and show that a strongly dependent type has finite
    stable weight. The key lemma for proving these results is Lemma
    \ref{lem:mutindisc} which says that under certain circumstances
    indiscernible sequences can be assumed to be mutually
    indiscernible. We find this interesting on its own.
    This section is related to more general works on different notions of
    weight in dependent theories: Onshuus and the author \cite{OnUs1} on dp-minimality, strong dependence and weight,
    \cite{OnUs2} on weight based on thorn-forking in rosy theories, and Adler \cite{Ad2} on
    ``burden''.

    The goal of the Appendix (which had originally been a part of section 3 and was removed for the sake of
    clarity) is to motivate viewing types as ultrafilters by passing to a more general framework
    of Keisler measures, in which Shelah's approach to finitely satisfiable types seems very natural.

    \medskip
    \subsection*{Acknowledgements}
    {The author thanks Saharon Shelah, Anand Pillay, Ita\"i Ben Yaacov and Hans Adler for fruitful
    discussions and correspondence. We are especially grateful to Assaf Hasson and Alf Onshuus
    for numerous conversations, questions, suggestions and comments on preliminary
    versions that motivated and improved this work. Many thanks to the anonymous referee
    for very helpful comments and suggestions.  }

\subsection{Notations}
In this paper, $T$ will denote a complete theory, $\tau$ will denote
the vocabulary of $T$, $L$ will denote the language of $T$. We will
assume that everything is happening in the monster model of $T$
which will be denoted by $\fC$. Elements of $\FC$ will be denoted
$a,b,c$, finite tuples will be denoted $\a,\b,\c$, sets (which are
all subsets of $\FC$) will be denoted $A,B,C$, and models of $T$
(which are all elementary submodels of $\FC$) will be denoted by
$M,N$, etc.

Given an order type $O$ and a sequence \inseq{\a}{i}{O}, we often
denote $\a_{<i} = \lseq{\a}{j}{i}$, similarly for $\a_{\le i}$,
$\a_{>i}$, etc. We will often identify a tuple $\a$ or a sequence
$\inseq{\a}{i}{O}$ with the set which is its union, but it should
always be clear from the context what we mean (although sometimes
when confusions might arise, we make the distinction, e.g.
$\Av(I,\cup I)$ will denote the average type of a sequence $I$ over
itself).

By $\a \equiv_A \b$ we mean $\tp(\a/A) = \tp(\b/A)$.

\subsection{Preliminaries}

Recall that a theory $T$ is called \emph{dependent} if there does
not exist a formula which exemplifies the independence property. We
are mostly going to use the following equivalent definition:

\begin{fct}\label{fct:depend}
    $T$ is dependent if and only if there do not exist an
    indiscernible sequence $I = \lseq{\a}{i}{\lam}$, a formula
    $\ph(\x,\y)$ and $\b$ such that both
    $$\set{i \colon \models \ph(\a_i,\b)}$$ and
    $$\set{i \colon \models \neg\ph(\a_i,\b)}$$
    are unbounded in $\lam$.
\end{fct}

A type $p \in \tS_m(B)$ is said to be \emph{definable} over $A$ if
for every formula $\ph(\x,\y)$ with $\len(\x) = m, \len(y) = k$
there exists a formula $d_p\x\ph(\x,\y)$ with free variables $\y$
such that for every $\b \in B^k$
\[
\ph(\x,\b) \in p \iff \models d_p\x\ph(\x,\b)
\]
A definition schema $d_p$ is said to be \emph{good} is for every set
$C$ the set
\[
\left\{\ph(\x,\c) \colon \ph(\x,\y) \; \mbox{is a formula}, \len(\x)
= m, \c \in C, \models d_p\x\ph(\x,\c) \right\}
\]
is a complete type over $C$. We call this type \emph{the free
extension} of $p$ to $C$ with respect to $d$ and denote it by
$p|^dC$. We call a type \emph{properly definable} over a set $A$ if
it is definable over $A$ by a good definition.

A type $p \in \tS(B)$ said to be \emph{finitely satisfiable} in a
set $A$ if for every formula $\ph(\x,\b) \in p$ there exists $\a \in
A$ such that $\models \ph(\a,\b)$. If $A \subseteq B$ then we also
say that $p$ is a \emph{coheir} of $p\rest A$.
    Clearly, if $p\in\tS(M)$ and $M$ is a model, then $p$ is finitely
    satisfiable in $M$.

Recall that a sequence $I = \inseq{\a}{i}{O}$ (where $O$ is a linear
ordering) is called \emph{indiscernible} over a set $A$ if the type
of $\a_{i_1},\ldots,\a_{i_k}$ over $A$ depends only on the order
between the indices $i_1, \ldots, i_k$ for every $k$. $I$ is called
an indiscernible \emph{set} if the type above depends on $k$ only.

A hyperimaginary element (tuple) $\a$ is said to be \emph{bounded}
over a set $A$ if the orbit of $\a$ under the action of
$\Aut(\fC^{heq}/A)$ is small, i.e. of cardinality less than $|\fC|$.
The \emph{bounded closure of} $A$, denoted by $\bdd^{heq}(A)$, is
the collection of all hyperimaginary elements bounded over $A$.
Clearly, this is a generalization of the algebraic closure, and
usually is a bigger set. If $T$ is stable, then $\bdd^{heq}(A) =
\acl^{eq}(A)$ for every set $A$. We will not make real use of
hyperimaginaries in the paper, hence will not concentrate on these
issues.

Let us say that two tuples are of \emph{Lascar distance $1$} over
$A$ if there exists an indiscernible sequence over $A$ containing
both tuples. Two tuples $\a$ and $\b$ are of \emph{Lascar distance
$k$} over $A$ if there exist $\a = \a_1, \a_2, \ldots, \a_{k+1} =
\b$ such that $\a_i,\a_{i+1}$ are of Lascar distance $1$ over $A$.
Recall that two tuples $\a,\b$ are said to have the same
\emph{Lascar strong type} over a set $A$ if they are of finite
Lascar distance over $A$. In this case we will often write
$\lstp(\a/A) = \lstp(\b/A)$.

\subsection{Global Assumptions}

    All theories mentioned in this paper are \emph{assumed to be
    dependent} unless stated otherwise. For the sake of
    clarity of presentation we also assume $T = T^{eq}$.


\section{Indiscernible sequences, nonsplitting and stationarity}

This section contains a collection of basic definitions and facts
some of which are well known, which will be used widely throughout
the paper.

Fact \ref{fct:depend} motivates the following definitions:

\begin{dfn}
    Let $I = \lseq{\a}{i}{\lam}$ be an indiscernible
    sequence, $B$ a set. We define the \emph{average type} of $I$
    over $B$, $\Av(I,B)$ as the set of all formulae $\ph(\x,\b)$
    such that $\set{i \colon \neg\ph(\a_i,\b)}$ is bounded in \lam.
\end{dfn}

\begin{rmk}
    If $I$, $B$ are as above, then $\Av(I,B) \in \tS(B)$.
\end{rmk}

Note that

\begin{rmk}
    Let $I = \lseq{\a}{i}{\lam}$ an indiscernible
    \emph{set}, $B$ a set. Then $\ph(\x,\b) \in \Av(I,B)$ if and
    only if $\set{i \colon \neg\ph(\a_i,\b)}$ is \emph{finite}.
\end{rmk}

In fact, we can say a bit more. The following definition is
motivated by \cite{Sh715}, Definition 1.7:

\begin{dfn}
  Let $I = \seq{\b_i} = \inseq{\b}{i}{O}$ be an infinite indiscernible sequence. We say that
  a formula $\ph(\x,\y)$ is \emph{stable} for $I$ if for every $\c \in \fC$
  the set $\set{i \in I \colon \ph(\b_i,\c)}$ is either finite or co-finite.
\end{dfn}

\begin{obs}\label{obs:stabform}
  If $I = \inseq{\b}{i}{O}$ is an infinite indiscernible set, then every $\ph(\x,\y)$
  is stable for $I$. Moreover, for every $\ph(\x,\y)$
  there exists $k = k_\ph < \om$ such that for every $\c \in \fC$, either
  $|\set{i \in O \colon \ph(\b_i,\c)}|<k$
  or   $|\set{i \in O \colon \neg\ph(\b_i,\c)}|<k$.
\end{obs}
\begin{prf}
  If not, by indiscernibility we have for every $U,W \subseteq I$ finite
  disjoint,
  $$\set{\ph(\b_i,\y) \colon i \in W}\cup\set{\neg\ph(\b_i,\y) \colon i \in U}$$
  is consistent, clearly contradicting dependence of $T$.
\end{prf}

It is often useful to consider $\Av(I,\cup I)$, i.e. the average
type of an endless indiscernible sequence over itself. Note that
$\ph(\x,\a_{<j}) \in \Av(I,\cup I)$ iff $\ph(\a_i,\a_{<j})$ holds
for all $i\ge j$. So:

\begin{rmk}\label{rmk:avind}
    Let $I$ be an indiscernible sequence. $\a \models \Av(I,\cup I)$
    if and only if $I^\frown\set{\a}$ is indiscernible.
\end{rmk}

Let us recall the definition of nonsplitting:

\begin{dfn}
\begin{enumerate}
\item
  A type $p \in \tS(B)$ \emph{does not split}
  over a set $A$ if whenever $\b,\c \in B$ have
  the same type over $A$, we have
  $\ph(\x,\b) \in p \iff \ph(\x,\c) \in p$ for every
  formula $\ph(\x,\y)$.
\item
  A type $p \in \tS(B)$ \emph{does not split strongly}
  over a set $A$ if whenever $\b,\c \in B$ are of Lascar distance $1$ over $A$, we have
  $\ph(\x,\b) \in p \iff \ph(\x,\c) \in p$ for every
  formula $\ph(\x,\y)$.
\item
  A type $p \in \tS(B)$ \emph{does not Lascar-split}
  over a set $A$ if whenever $\b,\c \in B$ have the same Lascar strong type over $A$, we have
  $\ph(\x,\b) \in p \iff \ph(\x,\c) \in p$ for every
  formula $\ph(\x,\y)$.
\end{enumerate}
\end{dfn}

Note that a global type doesn't split over a set $A$ if it is
invariant under the action of the automorphism group of \FC over
$A$.
One can also think of nonsplitting as a weak form of definability.

There are several ways to obtain types which do not split over a set
$A$.


\begin{obs} (No use of dependence)
\label{obs:nonsplit}
\begin{enumerate}
\item
  If a type over $B$ is finitely satisfiable in $A \subseteq B$,
  then it does not split over $A$.
\item
  If a type over $B$ is definable over $A \subseteq B$,
  then it does not split over $A$.
\end{enumerate}
\end{obs}

\begin{obs} (No use of dependence)\label{obs:Lascarstrong}
    Let $M$ be a $(|A|+\aleph_0)^+$-saturated model containing $A$,
    $p \in \tS(M)$.
    Then $p$ does
    not Lascar-split over $A$ if and only if $p$ does not split strongly over $A$.
\end{obs}
\begin{prf}
    One direction is clear. For the other one, if $p$ Lascar-splits
    over $A$, then there are $\b,\c\in M$ of the same Lascar strong
    type with $\ph(\x,\b)\land\neg\ph(\x,\c)\in p$. There are
    finitely many elements $\b = \b_1, \ldots,
    \b_k = \c$  of Lascar distance 1
    which witness that $\b,\c$ are of finite Lascar distance, and by
    saturation we may assume they all lie in $M$ (we may even assume that
    all the indiscernible sequences witnessing Lascar distance 1 are in $M$). Now if $p$ does
    not strongly split over $A$, then by induction $\ph(\x,\b_1) \in p \then
    \ph(\x,\b_i)\in p$ for all $i$,
    a contradiction
\end{prf}

Given a set $A$, there are boundedly many types which do not split
over $A$:

\begin{obs}\label{obs:fewnonsplit}(No use of dependence)
    Let $A$ be a set. Then there are at most $2^{2^{|A|+|T|}}$ types
    over $\FC$ which do not split over $A$. Same is true for
    splitting replaced with Lascar splitting or strong splitting.
\end{obs}
\begin{prf}
    Let $p$ be a global type which does not split over $A$.
    For every formula $\ph(\x,\c)$ with parameters, the answer to
    the question whether or not $\ph(\x,\c)$ belongs to $p$ depends
    only on the type $\tp(\c/A)$. Since there are at most $2^{|A|+|T|}$
    such types, we are done. For Lascar splitting use the same
    argument with types replaced with Lascar strong types (recall that the number
    of Lascar strong types over $A$ is also bounded by $2^{|A|+|T|}$ - e.g. Proposition 2.7.5 in \cite{Wa}); for
    strong splitting apply in addition Observation
    \ref{obs:Lascarstrong}.
\end{prf}

Recall that a formula \emph{forks} over a set $A$ if it implies a finite disjunction of formulae
each of which divides over $A$. A (partial) type forks over $A$ if it contains a forking formula.
In dependent theories forking is strongly related to splitting (see Fact \ref{fct:splitfork}); still, these
notions differ, so we need to state the analogue of Observation \ref{obs:nonsplit} separately:

\begin{obs} (No use of dependence)
\label{obs:nonfork}
\begin{enumerate}
\item
  If a (partial) type over $B$ is finitely satisfiable in $A \subseteq B$,
  then it does not fork over $A$. Moreover, if $\ph(\x,\b)$ is satisfiable
  in $A$, then $\ph(\x,\b)$ does not fork over $A$.
\item
  If a type $p$ over $B$ is definable over $A \subseteq B$,
  by a good definition $d_p$ then it does not fork over $A$.
\end{enumerate}
\end{obs}
\begin{prf}
    The first part is very easy. For the second part note that
    if $p$ forks over $A$, then every extension $q$ of $p$
    to a $(|A|+\aleph_0)^+$-saturated model $M$ containing $A$ divides
    over $A$; moreover, the indiscernible sequence exemplifying dividing
    lies in $M$. Clearly, taking $q=p|^d M$ we get a contradiction.
%


\end{prf}

The following observation due to Shelah (\cite{Sh783}, Observation 5.4) is easy
but extremely useful:

\begin{fct}\label{fct:strongsplit}
    In a dependent theory strong splitting implies dividing.
\end{fct}
\begin{prf}
    Assume $p \in \tS(B)$ splits strongly over $A$, that is, there
    exists a sequence $I = \lseq{\b}{i}{\om}$ indiscernible over $A$
    with $\ph(\x,\b_0),\neg\ph(\x,\b_1) \in p$; then
    $\psi(\x,\b_0\b_1) = \ph(\x,\b_0)\land\neg\ph(\x,\b_1) \in p$ divides
    over $A$, since the set
    $$ \set{\ph(\x,\b_{2i}),\neg\ph(\x,\b_{2i+1})\colon i<\om}$$
    is inconsistent by the dependence of $T$.
\end{prf}

The other implication is generally not true, as we will see later, unless one works
with types over slightly saturated models, in which case the following general fact
holds (it will not be of much importance to us):

\begin{fct} (No use of dependence)
    Let $A$ be a set, $M$ be a $(|A|+\aleph_0)^+$-saturated model containing $A$,
    $p \in \tS(M)$ which forks over $A$, then it splits strongly over $A$.
\end{fct}
\begin{prf}
    Easy.
\end{prf}



So (recalling also Observation \ref{obs:Lascarstrong}) we can
conclude:

\begin{fct}\label{fct:splitfork}
\begin{enumerate}
\item
    Let $M$ be a $(|A|+\aleph_0)^+$-saturated model containing $A$,
    $p \in \tS(M)$.
    Then $p$ does not split strongly over $A$  if and only if
    $p$ does not Lascar split over $A$ if and only if $p$ does
    not fork over $A$ if and only if $p$ does not divide over $A$.
\item
    Let $A$ be such that whenever $\b_1 \equiv_A \b_2$,
    then $\lstp(\b_1/A) = \lstp(\b_2/A)$ (e.g. $A$ is a model; one can show that in a dependent theory,
    assuming $A = \bdd^{heq}(A)$ and $\tp(\b_1/A)$ does not fork over $A$ is
    enough).

    Let $M$ be a $(|A|+\aleph_0)^+$-saturated model containing $A$,
    $p \in \tS(M)$.
    Then $p$ does not split over $A$  if and only if $p$ does
    not fork over $A$ if and only if $p$ does not divide over $A$.
\end{enumerate}
\end{fct}

One can find much information about the connections between
different ``pre-independence relations'' in dependent theories in
Adler \cite{Ad}.

\begin{cor}\label{cor:fewnonfork}
    There are boundedly many global types which do not fork over a given set $A$.
\end{cor}
\begin{prf}
    By Observation \ref{obs:fewnonsplit} and Fact \ref{fct:strongsplit}.
\end{prf}

\begin{rmk}
    Note that while Observation \ref{obs:fewnonsplit} is true in any
    theory, Corollary \ref{cor:fewnonfork} does not have to be true
    in a theory with the independence property, even if forking
    behaves nicely in it (e.g. the theory of the random graph). In
    fact, in a simple theory, a type over a model with a bounded
    number of global nonforking extensions is stationary, see e.g.
    \cite{Wa}, Lemma 2.5.15.
\end{rmk}

\begin{dfn}
\begin{enumerate}
\item
    Let $O$ a linear order, $A$ a set. We call a sequence $I =
    \inseq{\a}{i}{O}$ a \emph{nonsplitting/nonforking sequence over $A$} if it is
    an indiscernible sequence over $A$ of realizations of $p$ and
    $\tp(\a_i/A\a_{<i})$ does not split (respectively, fork) over $A$ for all $i \in O$.
\item
    If a sequence $I$ is indiscernible over $B$ and nonsplitting/nonforking
    over $A \subseteq B$, we sometimes say that $I$ is \emph{based
    on} $A$.
\item
    Let $p \in \tS(B)$ be a type. We call a sequence $I$
    a \emph{nonsplitting/nonforking sequence in} $p$ if it is a nonsplitting
    (respectively, nonforking)
    sequence over $B$ of realizations of $p$. We say that it is a
    sequence in $p$ \emph{nonsplitting/nonforking over $A$} (or based on $A$)
    if it is a sequence of realizations of $p$ indiscernible over
    $B$ and nonsplitting (respectively, nonforking) over $A$.
\end{enumerate}
\end{dfn}

The following fact is a well-known:
\begin{fct}\label{fct:splitind} (No use of dependence)
    Let $I = \lseq{\a}{i}{\lam}$ be such that
    \begin{itemize}
    \item
        $\tp(\a_i/A\a_{<i})$ does not split over $A$
    \item
        $\tp(\a_i/A\a_{<i}) = \tp(\a_j/A\a_{<i})$ for every $j \ge
        i $.
    \end{itemize}
    Then $I$ is a nonsplitting sequence over $A$ (that is, it is
    indiscernible over $A$).
\end{fct}

We will need the following slight modification of the Fact above.
We present the short proof for completeness.
\begin{obs}\label{obs:splitind} (No use of dependence)
    Let $I = \lseq{\a}{i}{\lam}$ be such that
    \begin{itemize}
    \item
        $\tp(\a_i/A\a_{<i})$ does not Lascar-split over $A$
    \item
        $\Lstp(\a_i/A\a_{<i}) = \Lstp(\a_j/A\a_{<i})$ for every $j \ge
        i $.
    \end{itemize}
    Then $I$ is a indiscernible over $A$.
\end{obs}
\begin{prf}
    The classical proof works, namely: we prove by induction on
    $k$ that $\Lstp(\a_{i_1}\ldots\a_{i_k}/A) =  \Lstp(\a_{j_1}\ldots\a_{j_k}/A)$
    for every $i_1<\ldots <i_k, j_1<\ldots<j_k$. For $k=1$ this is given.

    For $k>1$, assume wlog $j_k\ge i_k$. By the assumption
    $\Lstp(\a_{j_k}/A\a_{i_1}\ldots\a_{i_{k-1}}) = \Lstp(\a_{i_k}/A\a_{i_1}\ldots\a_{i_{k-1}})$.
    By the induction hypothesis $\Lstp(\a_{i_1}\ldots\a_{i_{k-1}}/A) =
    \lstp(\a_{j_1}\ldots\a_{j_{k-1}}/A)$ and by the lack of Lascar splitting
    $\Lstp(\a_{j_k}/A\a_{i_1}\ldots\a_{i_{k-1}}) = \Lstp(\a_{j_k}/A\a_{j_1}\ldots\a_{j_{k-1}})$,
    which completes the proof.
\end{prf}






\begin{obs}\label{obs:Avnonfork}
    Let $I = \lseq{\b}{i}{\om}$ be a nonforking sequence
    in $p \in \tS(A)$. Then $\Av(I,I\cup A)$ is a nonforking
    extension of $p$.
\end{obs}
\begin{prf}
    Let $\ph(\x,\b) \in \Av(I,I\cup A)$ (with $\b \in I$). By
    the definition of the average type, $\ph(\x,\b) \in \tp(\b_k/A\b_{<k})$
    for almost all $k<\om$. Since $I$ is nonforking, $\ph(\x,\b)$ does not fork
    over $A$.
\end{prf}





Fact \ref{fct:splitfork} shows that working over slightly saturated models provides
us with a very nice picture; unfortunately, this is not the case if one is interested in
types over sets (even models), as we shall see for instance in section 6 of the article.
This is why we need both nonforking and nonsplitting for slightly different purposes. As
the reader will see later, we believe that nonforking plays a deeper role.
 A major advantage of nonforking over nonsplitting is existence
of nonforking extensions, which is well-known and very useful:

\begin{fct}\label{fct:nonforkexist} (No use of dependence)
    Let $p$ be a partial type over a set $B$ which does not fork over $A \subseteq B$.
    Then there exists $p \in \tS(B)$ which does not fork over $A$.
\end{fct}

\begin{rmk}
    We will normally use Fact \ref{fct:nonforkexist} when $p \in \tS(A')$,
    $A \subseteq A' \subseteq B$.
\end{rmk}

It is natural to ask which types have existence and/or uniqueness of
nonsplitting extensions. The following general fact will become
useful later:


\begin{lem}\label{lem:unique} (No use of dependence)

  Assume $A \subseteq M$, $M$ is $(|A|+\aleph_0)^+$-saturated, and $p \in \tS(M)$
  does not split over $A$. Then for every $M \subseteq B$ there is a unique
  extension of $p$ to $B$ which does not split over $A$.
\end{lem}
\begin{prf}
    For existence, for every finite tuple $\b \in B$ introduce a tuple of variables $\y_\b$ of
    the same length and denote $r_\b(\y_\b) = \tp(\b/M)$.

    Let

    $$\Sigma = \bigcup_{\b \in B}r_\b(\y_\b)$$

    and

    $$\Gamma = p(\x) \cup \Sigma \cup
    \set{\ph(\x,\y_\b) \leftrightarrow \ph(\x,\y'_{\b'}) \colon \ph(\x,\y) \in L, \tp(\b/A) = \tp(\b'/A)}$$

    For every finite subset $B_0$ of $B$ find $B'_0
    \subseteq M$ satisfying the same type over $A$, and choose $\a'
    \models p\rest A\cup B'_0$. As $p \rest A\cup B'_0$ does not
    split over $A$, clearly $\ph(\a',\b') \leftrightarrow
    \ph(\a',\b'')$ for every $\b',\b'' \in B'_0$ satisfying the same
    type over $A$. This shows that $\Gamma$ is finitely satisfiable
    in $M$, and therefore consistent. By applying an automorphism
    over $M$, we are done.

  For uniqueness, let $\ph(\x,\y)$ be a formula and $\b \in B$.
  Let $\b' \in M$ realize $\tp(\b/A)$. Clearly, any nonsplitting extension
  of $p$ to $B$ chooses $\ph(\x,\b)$ if and only if $\ph(\x,\b')\in p$.
\end{prf}


Note that even the existence in the lemma above can not be taken for
granted (if we do not work over an $|A|$-saturated model), even if
$T$ is dependent, $A$ itself is a (saturated) model, and $p\rest A$
is generically stable. See more in Discussion \ref{dsc:defsplit} and Example
\ref{exm:FDO}.

Another case of uniqueness of nonsplitting extensions occurs for average types:

\begin{lem}\label{lem:seqstat}
    Let $I = \lseq{\b}{i}{\om}$ be an indiscernible set over $A$ which is also a
    nonsplitting sequence. Denote $p = \Av(I,A\union I)$.
    Assume that $q$ is a global extension of $p$ which does not split over $A$. Then
    $q = \Av(I,\fC)$.
\end{lem}
\begin{prf}


    Denote $B = A\union I$.
    Let $\ph(\x,\bar c) \in
    q$, and assume towards contradiction $\neg\ph(\x,\bar c)\in
    \Av(I,\fC)$, so
    \begin{equ}\label{equ:seqstat:1}
    $\neg\ph(\b_i,\bar c)$ holds for almost all $i<\om$.
    \end{equ}
     Let $J  = \lseq{\b'}{i}{\om} \subseteq M$ be a nonsplitting sequence in $q$ over $BI\bar c$
     (that is, choose $\b'_i \models q\rest BI\bar c\b'_{<i}$).
     Clearly
     \begin{equ}\label{equ:seqstat:2}
     $\ph(\b'_i,\bar c)$ holds for all $i<\om$.
     \end{equ}
    \begin{clm}
        $I^\frown J$ is indiscernible.
    \end{clm}
    \begin{proof}
        Since $I$ is a nonsplitting sequence and $J$ is nonsplitting
        over $BI$, both based on $A$,
        it is
        enough to show that for every $i,j<\om$ $\b_i
        \equiv_{A\b_{<i}} \b'_j$, see Fact \ref{fct:splitind}. But this is also clear as $\b'_j \models \Av(I,A\union
        I)$ and $I$ is indiscernible over $A$.
    \end{proof}

    Combining all of the above, since $I$ is an indiscernible set, we clearly get a contradiction to
    dependence.

\end{prf}

Following the lemma above, one might want to define stationary types
as those having a unique nonsplitting extension over $\fC$. We will
see later (e.g. Discussion \ref{dsc:defsplit}) that this definition
is wrong, even for generically stable types, one reason being
precisely that nonsplitting types do not have to have global
nonsplitting (invariant) extensions. Therefore nonforking gives rise
to a better notion of stationarity.


\begin{dfn}
    We call a type $p \in \tS(A)$ \emph{stationary} if it has a
    unique \emph{nonforking} extension to any superset of $A$.
\end{dfn}

Let us prove an analogue of Lemma \ref{lem:seqstat} for nonforking.
It is probably the central result of this section.

First we need to ``improve'' Fact \ref{fct:strongsplit} slightly
adjusting it to our purposes. Note that we weaken both the
assumption and the conclusion (but forking in the conclusion is
really all we need).

\begin{obs}\label{obs:lascarsplit}
    Lascar splitting implies forking.
\end{obs}
\begin{prf}
    Let $p \in \tS(B)$ Lascar split over $A$, and assume it does not
    fork over $A$. By Fact \ref{fct:nonforkexist} there exists a
    global type $q$ extending $p$ which does not fork over $A$.
    Being an extension of $p$, it clearly Lascar splits over $A$, hence strongly splits by
    Observation \ref{obs:Lascarstrong}; a contradiction to Fact
    \ref{fct:strongsplit}.
\end{prf}


\begin{lem}\label{lem:globalnonfork}
\begin{enumerate}
\item
    Let $I = \lseq{\b}{i}{\om}$ be a nonforking sequence over $A$ which is also an
    indiscernible set over $A$. Denote $p = \Av(I,A\union I)$.
    Let $p^*$ be a global extension of $p$ which does not fork over $A$. Then
    $p^* = \Av(I,\fC)$.
\item
    Let $I$ be a nonforking sequence over $A$ which is an indiscernible set over $A$.
    Then $\Av(I,\fC)$ does not fork over $A$.
\end{enumerate}
\end{lem}
\begin{prf}
The second part follows from the first since $\Av(I,A\cup I)$ does
not fork over $A$ by Observation \ref{obs:Avnonfork}, hence can be
extended to a global type which does not fork over $A$.

    For the first part, we are going to repeat the
    proof of Lemma \ref{lem:seqstat} replacing
    splitting with Lascar-splitting. Denote $B = A\union I$.
    Let $\ph(\x,\bar c) \in
    q$, and assume towards contradiction $\neg\ph(\x,\bar c)\in
    \Av(I,\fC)$, so
    \begin{equ}
    $\neg\ph(\b_i,\bar c)$ holds for almost all $i<\om$.
    \end{equ}
     Let $J  = \lseq{\b'}{i}{\om} \subseteq M$ be a nonsplitting sequence in $q$ over $B\bar c$
     (that is, choose $\b'_i \models q\rest B\bar c\b'_{<i}$).
     Clearly
     \begin{equ}
     $\ph(\b'_i,\bar c)$ holds for all $i<\om$.
     \end{equ}
    \begin{clm}
        $I^\frown J$ is indiscernible.
    \end{clm}
    The claim clearly suffices.

    In order to prove the claim, we will have to be a bit more
    careful than in Lemma \ref{lem:seqstat} and apply Observation
    \ref{obs:splitind}. So we have to argue that the sequence is
    Lascar-nonsplitting and Lascar strong type of an element over
    the previous ones is ``increasing''. Lascar-nonsplitting
    follows from nonforking by Observation \ref{obs:lascarsplit}. $I$ is an
    $A$-indiscernible sequence, so clearly Lascar strong type of an
    element is increasing, same for $J$. So it is again
    enough to show that for every $i,j<\om$ $\lstp(\b_i/A\b_{<i}) = \lstp(\b'_j/A\b_{<i})$.
     But $\b'_j \models \Av(I,A\union I)$, so it continues
    $I$; hence $\b'_j$ and $\b_i$ are of Lascar-distance 1 over
    $A\b_{<i}$.
\end{prf}

The following definition is standard:

\begin{dfn}
    We call two types $p$ and $q$ \emph{parallel} if they have a common
    nonforking extension; that is, if there exists a type $r$ which
    is a nonforking extension of both $p$ and $q$.
\end{dfn}

Since we'll be working a lot with definable types, the notion of a Morley sequence with respect
to a given definition will come handy:

\begin{dfn}\label{dfn:Morley}
  Let $p \in \tS(B)$ be a type definable over $A \subseteq B$ by a
  definition schema $d_p$, $O$ an order type. Then $I = \inseq{\a}{i}{O}$
  is called a \emph{Morley sequence} in $p$ over $B$ \emph{based on} $A$
  (with respect to the definition schema $d_p$) if for every
  $i \in O$ we have $\a_i \models p |^d B_i$, where
  $B_i = B\cup\set{\a_j \colon j<i}$ as usual.
\end{dfn}

The following is pretty clear:

\begin{obs}\label{obs:morleyset}
    Let $I$ be a Morley sequence in $p$ over $B$ with respect to the definition schema
    $d_p$, and assume furthermore that $I = \inseq{\a}{i}{O}$ is an indiscernible set. Then
    for every $i \in O$ we have $\a_i \models p|^d B\cup\set{\a_j \colon j\neq i}$.
\end{obs}


\section{Finitely satisfiable types}

In this section we develop some basic theory of finitely satisfiable
types. Notions introduced here are essential for understanding
Section 5, but a reader who is not interested in Shelah's approach
to ``stable'' types can easily skip this section in the first
reading and proceed to the next one, where the general theory of
generically stable types is developed. Those readers would like to
see the connection between the two approaches and intend to read
this section, are encouraged to also have a look at the Appendix,
where we try to motivate viewing types as ultrafilters by passing to
the space of measures on the Boolean algebra of definable sets
(Keisler measures).


\begin{dfn}
  Let $A, B$ sets.
      We denote the boolean algebra of $B$-definable subsets of $A^m$ by $\Def_m(A,B)$.
      Let $\Def(A,B) = \bigcup_{m<\om}\Def_m(A,B)$. We omit $B$ if $B=A$.
\end{dfn}

\begin{rmk}
  Note that an ultrafilter \FU on $\Def(\fC,B)$ precisely
  corresponds to
  a complete type over $B$. In order to
  be consistent with Shelah's notions and terminology,
  we call this type the \emph{average type} of \FU and denote
  it by $\Av(\fU,B)$.
\end{rmk}

\begin{dfn}
\label{dfn:average}

  Let $A\neq\emptyset$ and $B$ be sets, $\fU$ an ultrafilter on $\Def_m(A,B)$.
  We define
  the \emph{average type} of $\fU$ over $B$ by
  \[
     \Av(\fU,B) = \left\{\ph(\x,\b) \colon \b \in B \; \mbox{and} \;\set{\a \in A^m \colon \ph(\a,\b)} \in \fU \right\}
  \]
\end{dfn}

\begin{obs}
  For $A,B,\fU$ as above, $\Av(\fU,B) \in \tS_m(B)$ finitely satisfiable in $A$ (so if
  $A\subseteq B$, then $\Av(\fU,B)$ is a coheir
  of its restriction to $A$).
\end{obs}
\begin{prf}
  For $\ph(\x,\b)$, a formula with $m$ free
  variables over $B$, either \set{\a \in A^M \colon \neg\ph(\a,\b)} or
  \set{\a \in A^M \colon \ph(\a,\b)} is in $\fU$, so the average type is complete.
  Finite satisfiability in $A$ (hence consistency) is also
  clear.
\end{prf}

\begin{rmk}
    Note that in the definition of $p = \Av(\fU,B)$ above, the set $A$
    in which $p$ is finitely satisfiable is given by $\fU$. We can also forget $A$
    sometimes, since as a
    complete type over $B$, $p$ does not depend on $A$ in the
    following sense:
    if $A \subseteq A'$, $\fU, \fU'$ ultrafilters on
    $\Def(A,B)$ and
    $\Def(A',B)$ respectively, such that $\fU = \fU'\rest\Def(A,B)$
    in the obvious sense,
    then $\Av(\fU,B) = \Av(\fU',B)$.
\end{rmk}


\begin{obs}\label{obs:fs}
  Let $A,B$ be sets, $p \in \tS_m(B)$. Then $p$ is finitely satisfiable in $A$
  if and only if for some ultrafilter $\fU$ on $\Def_m(A,B)$ $p = \Av(\fU,B)$.
\end{obs}
\begin{prf}
  If $p = \Av(\fU,B)$ for some ultrafilter $\fU$ on $A^m$, then clearly $p$ is finitely
  satisfiable in $A$. On the other hand, if $p$ is finitely satisfiable in $A$, it is
  easy to see that the collection $\set{\ph^\fC(\x,\b)\cap A^m \colon \ph(\x,\b) \in p}$
  is an ultrafilter on $\Def_m(A,B)$.
\end{prf}

\begin{dfn}
\label{dfn:shseq}
\begin{enumerate}
\item
  Let $A \subseteq B$, $O$ an order type.  We say that a sequence  $I = \inseq{\a}{i}{O}$ is a \emph{Shelah sequence}
  over $B$ \emph{supported on} $A$ if (denoting  $B_i = B\cup\set{\a_j\colon j<i}$)
  \begin{itemize}
  \item
    $\tp(\a_i/B_i)$ is finitely
    satisfiable in $A$
 \item
   $I$ is an indiscernible sequence over $B$
  \end{itemize}
\item
  Let $A \subseteq B$, $p \in \tS(B)$ finitely satisfiable in $A$.
  We call a sequence $I$ a \emph{Shelah sequence}
  in $p$ supported on $A$ if it is a Shelah sequence over $B$ supported on $A$
  of realizations of $p$.
\end{enumerate}
\end{dfn}

\begin{lem}
\begin{enumerate}
\item
  If $p \in \tS^m(B)$ finitely satisfiable in $A$, then there is an infinite Shelah sequence
  in $p$ over $B$ supported on $A$. Moreover, for every ultrafilter $\fU$ on
  $\Def_m(A,\fC)$ satisfying
  $\Av(\fU,B)=p$ (see Observation \ref{obs:fs}),
  the sequence defined by $\a_i \models \Av(\fU,B\cup\lseq{\a}{j}{i})$ is such a sequence.
\item
  If $p \in \tS^m(B)$ finitely satisfiable in $A$, then there is a Shelah sequence
  in $p$ over $B$ supported in $A$ of any order type.
\item
   \lseq{\a}{i}{\lam} is a Shelah sequence over $B$ based on $A$ if and only
   if for some ultrafilter $\fU$ on $\Def(A,\fC)$, $\a_i \models \Av(\fU,B\lseq{\a}{j}{i})$.
\end{enumerate}
\end{lem}
\begin{prf}
\begin{enumerate}
\item
    Denote $B_i = B\cup\lseq{\a}{j}{i}$.

  It is clear that any sequence obtained in this way (say, of length \lam) satisfies:
  \begin{itemize}
  \item
    $\tp(\a_i/B_i)$ is finitely
    satisfiable in $A$
  \item
   for $k\le j<i$, $\tp(\a_j/B_k) = \tp(\a_i/B_k)$
  \end{itemize}
  Clearly by Observation \ref{obs:nonsplit} $\tp(\a_i/B_i)$ does not split over $A$. Now by
  Fact \ref{fct:splitind}
  \lseq{\a}{i}{\lam} is an
  indiscernible sequence.
\item
  By the previous clause and compactness.
\item
  Let \lseq{\a}{i}{\lam} be a Shelah sequence. Denote
  $B_i = B\cup\lseq{\a}{j}{i}$.
  By Observation \ref{obs:fs}, for every $i$ there exists
  $\fU_i$ an ultrafilter on $\Def(A,B_i)$
  such that $\tp(\a_i/B_i) = \Av(\fU_i,B_i)$.
  Note that for $i<j$, $\fU_i \rest \Def(A,B_j) =
  \fU_j$. Let

  $$\fU_\lam = \bigcup_{i<\lam}\fU_i$$

  Then $\fU_\lam$ is a pre-filter on $\Def(A,B_\lam)$, in particular
  on $\Def(A,\fC)$. Extending it to an ultrafilter, we are done.

  In other words: the union of types $\tp(\a_i/B_i)$ is a partial type
  over $B_\lam$ finitely satisfiable in $A$, so it can be extended to
  a type $q$ over $B_\lam$ finitely satisfiable in $A$. Now let $\fU$ be such
  that $q = \Av(\fU,B_\lam)$.
\end{enumerate}
\end{prf}

We would like the reader to compare the definition of a Shelah
sequence to Definition \ref{dfn:Morley}. We will see later that
generally Shelah sequences and Morley sequences are not the same
object, even if both exist.

\begin{rmk}
\begin{enumerate}
  \item A Shelah sequence in $p \in \tS(B)$ supported in $A$ is a
  nonsplitting sequence in $p$ based on $A$.
  \item A Morley sequence in $p \in \tS(B)$ based on $A$ is a
  nonsplitting sequence in $p$ based on $A$.
\end{enumerate}
\end{rmk}

\begin{dfn}
\label{dfn:unique}
  Let $p\in \tS(B)$ finitely satisfiable in $A \subseteq B$. We say that $p$ has \emph{uniqueness}
  over $A$ if there is a unique type of a Shelah sequence in $p$ based on $A$.
\end{dfn}

\begin{rmk}
    Stationary types have uniqueness.
\end{rmk}


\begin{obs}\label{obs:unique}
  Let $p \in \tS(M)$ finitely satisfiable in $A \subseteq M$, $M$ is $|A|^+$-saturated.
  Then $p$ has uniqueness over $A$.
\end{obs}
\begin{prf}
  By Lemma \ref{lem:unique}.
\end{prf}

In \cite{Sh715} Shelah shows the following:

\begin{fct}\label{fct:defunique}
  Let $p\in \tS(A)$ finitely satisfiable in $A$ and definable over $A$. Then $p$ has
  uniqueness.
\end{fct}

Note that $A$ is just a set, so there is no reason why there would
be only one nonforking (or nonsplitting) extension of $p$ to an
arbitrary superset; in fact, this is generally false, see Example
\ref{exm:unstable} below. Moreover, it is generally false that a
Shelah sequence and a Morley sequence in $p$ obtained from extending
$p$ by its definition over $A$ have the same type. Still, there is a
unique Shelah sequence.

\section{Generically stable types}

In this section we propose an approach to generically stable types different
from \cite{Sh715} which does not require working with finitely
satisfiable types. The definition below is more general and might
seem weaker than the on given by Shelah, but as it turns out, they
give rise to the same notion. See section 5 for more details.


\begin{dfn}
    We call a type $p\in \tS(A)$ \emph{generically stable} if there exists a
    nonforking sequence \lseq{\b}{i}{\om} in $p$ (over $A$) which is an
    indiscernible set.
\end{dfn}


\begin{lem}\label{lem:Idef}
    Let $I = \lseq{\b}{i}{\om}$ be an indiscernible set over a set $A$, $C\supseteq A$. Then $p =
    \Av(I,C)$ is definable over $\cup I$.
\end{lem}
\begin{prf}

    Let $\ph(\x,\y)$ be a formula and let $k = k_\ph$ be as in
    Observation \ref{obs:stabform}. Now clearly for every $\c \in C$

    $$ \ph(\x,\c) \in \Av(I,C)$$
    if and only if
    $$|\set{i<2k \colon \models \ph(\b_i,\c)}| \ge k$$
    if and only if
    $$\bigvee_{u\subset
    2k,|u|=k}\bigwedge_{i\in u}\ph(\b_i,\c) $$

    So $p$ is definable over $I$ by the schema
    $$d_p\x\ph(\x,\y) = \bigvee_{u\subset
    2k_\ph,|u|=k_\ph}\bigwedge_{i\in u}\ph(\b_i,\y) $$
    as required.
\end{prf}

\begin{lem}\label{lem:almostover}
    Let $p\in \tS(A)$ be generically stable. Then $p$ is properly definable (definable by a good definition) almost over
    $A$.
\end{lem}
\begin{prf}
    Let $I = \lseq{\b}{i}{\om}$ be a nonforking indiscernible (over $A$) set in
    $p$.

    Let $\ph(\x,\y)$ be a formula, then $p$ is definable over $I$ as in Lemma
    \ref{lem:Idef} by
    $$\vartheta(\y,\b_{<2k}) = d_p\x\ph(\x,\y) = \bigvee_{u\subset
    2k_\ph,|u|=k_\ph}\bigwedge_{i\in u}\ph(\b_i,\y) $$.

    \begin{clm}
        $\vartheta(\x,\b_{<2k})$ as above is almost over $A$.
    \end{clm}

    Note that once we have proven the Claim we are done: $p$ is
    definable almost over $A$ by a definition which is clearly good
    (it defines $\Av(I,\fC)$).

    For the proof of the Claim note that otherwise we would have
    unboundedly many pairwise nonequivalent automorphic copies of $\vartheta$ over
    $A$. In other words, we would have an unbounded sequence of
    automorphisms $\seq{\sigma_\al}$ over $A$ such that
    $\set{\vartheta_\al = \sigma_\al(\vartheta)}$ are pairwise
    nonequivalent. Let $I_\al = \sigma_\al(I)$, $p_\al =
    \Av(I_\al,I_\al\cup A)$. By Lemma \ref{lem:globalnonfork}
    $q_\al = \Av(I_\al,\fC)$ all do not fork over $A$. Note that
    $q_\al$ is definable by $\vartheta_\al$ and therefore are all distinct.
    So $\seq{q_\al}$ is an unbounded sequence
    of global types all of which do
    not fork over $A$, which is a contradiction to
    Corollary \ref{cor:fewnonfork}.

\end{prf}


   Note that all we used in the proof of Lemma \ref{lem:almostover} is
    that there exists an indiscernible set $I$ in $p$ such that
    $\Av(I,\fC)$ does not fork (or split) over $A$.
So the following is a corollary of the proof of Lemma
\ref{lem:almostover}:

\begin{cor}\label{cor:almostover}
    Let $p \in \tS(A)$, $I$ an indiscernible set in $p$ such
    that $\Av(I,\fC)$ does not fork/split over $A$. Then $p$ is
    properly definable almost over $A$.
\end{cor}

Let is summarize the Lemmas above:

\begin{cor}\label{cor:def}
    Let $p \in \tS(A)$ be generically stable, $I$ a nonforking indiscernible set in $p$. Then
    there exists a definition schema $d_p$ over $I$, almost over $A$, such that
    for every set $C$, $\Av(I,C\cup I) = p|^d (C\cup I)$.
\end{cor}

This allows us to speak about definitions and free extensions
instead of averages, which makes our lives quite a bit simpler. One
important consequence is stationarity of generically stable types.
Recall that we defined stationarity using \emph{nonforking}.

First, let us recall and slightly rephrase Lemma
\ref{lem:globalnonfork}:

\begin{cor}\label{cor:avstat}
    Let $p \in \tS(A)$ be generically stable, $I$ a nonforking indiscernible set in $p$. Then $\Av(I,A\cup I)$ has
    a unique extension to $\fC$ which does not fork over $A$. This
    extension equals $\Av(I,\fC)$.
\end{cor}

Now we proceed to the main stationarity result.


\begin{prp}\label{prp:defstat}
    Let $p \in \tS(A)$ be a generically stable type witnessed by a nonforking indiscernible set $I$
    such that the definition schema $d_p$ as in Corollary \ref{cor:def} is \emph{over}
    $A$ (e.g. $A = \acl(A)$). Then $p$ is stationary.
\end{prp}
\begin{prf}
    We aim to show that $p$ has a unique nonforking extension to any superset of $A$.
    By existence of nonforking extensions (Fact
    \ref{fct:nonforkexist})
    and Corollary \ref{cor:avstat}, it is enough
    to show that the only nonforking extension of $p$ to $A\union I$ is
    $\Av(I,A\cup I)$. By Fact \ref{fct:strongsplit} it is enough to show
    that $\Av(I,A\cup I)$ is the only extension of $p$ to $A\cup I$ which does
    not \emph{split strongly} over $A$.
    Denote $B = A\union I$, $B_k = A\union \lseq{\b}{i}{k}$ for
    $k\le\om$.

    Let $\b' \models p$, $\tp(b'/B)$ does not split strongly over $A$. We
    show by induction on $k$ that $\tp(\b'/B_k) = \Av(I,B_k)$.
     There is nothing to show for $k = 0$.

    Assume the claim for $k$, and suppose $\ph(\b',\b_0, \ldots, \b_k,
    \a)$ holds. Let $\ps(\x,\b_{<k},\b',\a) =
    \ph(\b',\b_0,\ldots,\b_{k-1},\x,\a)$, so

    \begin{equ}\label{equ:A.1}
    $\ps(\b_k,\b_{<k},\b',\a)$ holds.
    \end{equ}

    Note that since $\tp(\b'/B)$ doesn't split strongly over $A$,
    the set $\seq{\b_i\colon i\ge k}$ is indiscernible over
    $B_k\b'$: for every $i_1,\ldots,i_\ell$ and $j_1,\ldots,j_\ell$
    all greater or equal to $k$, we have
    $$\lstp(\b_{i_1}\ldots\b_{i_\ell}/B_k) =
    \lstp(\b_{j_1}\ldots\b_{j_\ell}/B_k) $$
    (moreover, their Lascar distance is 1) and therefore
    by the lack of strong splitting
    $$\b'\b_{i_1}\ldots\b_{i_\ell} \equiv_{B_k}
    \b'\b_{j_1}\ldots\b_{j_\ell} $$
    which precisely means
    $$\b_{i_1}\ldots\b_{i_\ell} \equiv_{B_k\b'}
    \b_{j_1}\ldots\b_{j_\ell} $$

    So by ~\ref{equ:A.1} we see that $\ps(\b_\ell,\b_{<k},\b',\a)$
    holds for all $\ell$ big enough, and therefore

    \begin{equ}
    $\ps(\x,\b_{<k},\b',\a) \in \Av(I,B\b')$.
    \end{equ}

    Therefore (denoting $q = \Av(I,\fC)$), $d_q\x\ps(\x,\b_{<k},\b',\a)$ holds, where the
    definition is over $A$ (by the assumptions of the Proposition).
    So we get $\theta(\y) =
    d_q\x\ps(\x,\b_{<k},\y,\a)$ is in $\tp(\b'/B_k)$ and therefore (by
    the induction hypothesis) is in $\Av(I,B_k)$, which we think now as of a type in $\y$.
    This means that
    $d_q\x\ps(\x,\b_{<k},\b_\ell,\a)$ holds for almost all $\ell$, and therefore
    (since $d_q$ defines $\Av(I,\fC)$) we have $\ps(\x,\b_{<k},\b_\ell,\a) \in
    \Av(I,B)$ for almost all $\ell$. Let $\ell$ be
    such, so by the definition of average type, there exists an
    $m$ such that $\ps(\b_m,\b_{<k},\b_\ell,\a)$, that is,
    $\ph(\b_\ell,\b_{<k},\b_m,\a)$ holds. Since
    $I$ is an indiscernible set, we get
    $\ph(\b_m,\b_{<k},\b_k,\a)$  for all $m$ big enough, and
    therefore
    \begin{equ}
    $\ph(\x,\b_{\le k},\a) \in \Av(I,B)$
    \end{equ}
    as required.

\end{prf}

\begin{cor}\label{cor:stabstat}
    A generically stable type over an algebraically closed set is stationary.
\end{cor}


\begin{dsc}\label{dsc:defsplit}
    From examining the proofs it might seem like we have shown that a generically stable type $p$
    over an algebraically closed set $A$ has a unique
    nonsplitting (or not strongly splitting) extension over any set, and therefore in particular
    every nonsplitting sequence in $p$
    is an indiscernible set, etc; but this is not the case. The reason is that
    if $I = \lseq{\b}{i}{\om}$ is a nonsplitting indiscernible set in $p$,
    a nonsplitting extension of $p$ to $A\b_0$ does not need have an extension
    over $I$ which does not split over $A$. We will come back to this phenomenon in
    section 6 while discussing examples of generically stable types. Let us formulate precise statements
    that do follow from the analysis above:
\end{dsc}

\begin{cor}\label{cor:uniquesplit}
    Let $p \in \tS(A)$ be a generically stable type which is definable over $A$,
    $C \supseteq A$ is a set containing an
    infinite Morley sequence $I$ in $p$ (or just a nonforking
    sequence in $p$ which is an indiscernible set). Then $p$
    has a unique extension to $C$ which does not split strongly over $A$. This extension
    equals $q=\Av(I,C)$.

    In particular, $p$ has
    a unique extension to any $(|A|+\aleph_0)^+$-saturated model
    $M$ containing $A$ which does not split strongly (equivalently, Lascar split) over $A$.
    This unique extension is definable over $A$ and equals $p|^dM$.

    If $A$ is e.g. a model (can be weakened to $A = \bdd^{heq}(A)$)
    then strong splitting above can be replaced with splitting.
\end{cor}

The following is an easy consequence of stationarity:

\begin{cor}
    A nonforking extension of a generically stable type is generically stable.
\end{cor}
\begin{prf}
    Clearly, every extension of a generically stable type $p \in \tS(A)$ to the algebraic
    closure of $A$ is generically stable; now use stationarity.
\end{prf}

Recall that we call two types \emph{parallel} if they have a common
nonforking extension.

\begin{lem}\label{lem:parallel}
    Let $p \in \tS(A)$ be a type, then $p$ is generically stable if and only if
    every $q$ parallel to $p$ is generically stable.
\end{lem}
\begin{prf}
    Let $q \in \tS(B)$ be parallel to $p$ and let $r \in \tS(C)$
    witness this, that it, $A,B \subseteq C$, $r$ extends $p$ and
    $q$ and does not fork over both $A$ and $B$. Without loss of
    generality $C = \acl(C)$.

    By the previous Corollary, $r$ is generically stable. Fix $M$ a
    $(|C|+|T|)^+$-saturated model containing $C$. Let $I$ be
    a nonforking sequence (set) in $r$ contained in $M$, then $r$ has a
    unique nonforking extension to $M$ which equals $\Av(I,M)$.

    Since $r$ does not fork over $B$, by Fact \ref{fct:nonforkexist} there exists
    $q^* \in \tS(M)$ extending $r$ which does not fork over $B$.
    Clearly $q^*$ does not fork over $C$ and therefore $q^* =
    \Av(I,M)$.

    So we've shown that $\Av(I,M)$ does not fork over $B$; applying
    Corollary \ref{cor:almostover}, $\Av(I,M)$ is definable almost over $B$,
    so $I$ is a Morley (and therefore nonforking) sequence in $q$ which is an
    indiscernible set, hence $q$ is generically stable.
\end{prf}

\begin{cor}(Transitivity of forking for generically stable types)
    Let $p \in \tS(C)$, $A \subseteq B \subseteq C$, one of
    $p, p\rest B, p \rest A$ is generically stable, $p$ does not fork over
    $B$ and $p\rest B$ does not fork over $A$. Then all of the above
    three types are generically stable and $p$ does not fork over $A$.
\end{cor}
\begin{prf}
    Easy at this point.
\end{prf}

Recall that by a well-known result of Kim, transitivity of forking
implies simplicity of $T$, therefore one can't expect forking to be
transitive in general in a dependent unstable theory.

Note that combining all the results of this section one can quite
easily deduce properties of stable independence relation based on
forking for realizations of generically stable types; we will come back to this
issue in section 7.





\section{Generically stable types - Shelah's approach}

The following definition is given in \cite{Sh715}:

\begin{dfn}
  A type $p\in \tS(B)$ is called \emph{Shelah-stable} if
  there exists an infinite Shelah sequence in $p$
  which is an indiscernible set.
\end{dfn}

\begin{rmk}
  Note that in particular $p$ is finitely satisfiable in $B$.
\end{rmk}

Shelah shows in \cite{Sh715} that:

\begin{fct}\label{fct:stabdef}
  If $p$ is a Shelah-stable type, then $p$ is definable, hence has uniqueness, so every Shelah sequence
  in it is an indiscernible set.
\end{fct}

A natural particular case of a finitely satisfiable type is a type
over a model. The following lemma will help us understand
Shelah-stable types over slightly saturated models:

\begin{lem}\label{lem:stabtype}
  Let $M$ be $|A|^+$-saturated, and $p\in\tS(M)$ finitely satisfiable in $A$. Assume
  furthermore that $p$ is definable over $A$. Then $p$ is Shelah-stable.
\end{lem}
\begin{prf}
  Let \lseq{\a}{i}{\om} a Morley sequence in $p$ based on $A$ (recall that $p$ is definable over
  $A$). Clearly it is a nonplitting sequence. Since by Lemma \ref{lem:unique} $p$ has a unique
  extension to any superset of $M$ which does not split over $A$, \lseq{\a}{i}{\om} is also
  a Shelah sequence.

  Let us show for instance that $\a_0\a_1 \equiv \a_1\a_0$, and even
  $\a_0\a_1 \equiv_M \a_1\a_0$.
  Since the type $\tp(\a_1/Ma_0)$ is an heir of $p$ (as it is definable by the same definition
  scheme), and since $M$ is a model, it follows that the type $\tp(\a_0/M\a_1)$ is a
  co-heir of $p$; moreover it is \emph{the} coheir of $p$ since $p$
  has uniqueness by Observation \ref{obs:unique}. In other words, $\a_1\a_0$ start a
  Shelah sequence in $p$, and by uniqueness $\tp(\a_1\a_0/M) = \tp(\a_0\a_1/M)$, as required.
\end{prf}


In Lemma \ref{lem:stabtype} it is necessary to assume that $M$ is
saturated over $A$. In general it is not true that if $p$ is both
definable over $A$ and finitely satisfiable in it then $p$ is
Shelah-stable:

\begin{exm}\label{exm:unstable}
  Let $T$ be the theory of dense linear orderings with no endpoints, $A = \setQ$, $p \in \tS(A)$
  ``the type at infinity'', i.e. $[x>a] \in p$ for all $a \in A$. Then $p$ is clearly finitely
  satisfiable in $A$ and definable over $A$ (in fact, over the empty set),
  but is not Shelah-stable. It still has uniqueness, and
  any Shelah sequence in $p$ based on $A$ is simply a descending sequence. A Morley sequence
  in $p$ is, on the other hand, an increasing sequence.
\end{exm}

So in order to obtain an ``if and only if'' criterion, we will have to strengthen
Fact \ref{fct:stabdef} slightly (the proof can be extracted from Shelah's proof
of Fact \ref{fct:stabdef}):

\begin{thm}\label{thm:stabtype}
Let $p \in \tS(A)$ finitely satisfiable in $A$. Then the following are equivalent:
\begin{enumerate}
\item
  $p$ is Shelah-stable
\item
  Any coheir of $p$ over $\fC$ is definable over $A$
\item
  Some coheir of $p$ over some $(|A|+|T|)^+$-saturated model containing
  $A$ is definable over $A$
\end{enumerate}
\end{thm}
\begin{prf}
$(ii) \then (iii)$ is clear and $(iii) \then (i)$ follows from Lemma
\ref{lem:stabtype}. So we only need to prove $(i) \then (ii)$. So
assume $p$ is Shelah-stable, and let $q$ be a coheir of $p$ over
$\fC$; then $q = \Av(\fU,\fC)$ for some ultrafilter $\fU$ on
$\Def(\fC,A)$.

Let \lseq{\b}{i}{\om} be a $\fU$-Shelah sequence in $p$ over $B$.
Since $p$ is Shelah-stable, \seq{\b_i} is an indiscernible set. Let
$\ph(\x,\y)$ be a formula, and let $\Delta$ be a finite set of
formulae containing $\exists \y \ph(\x,\y)$. Let $k = k_\ph < \omm$
be as in Observation \ref{obs:stabform}.

Since $p$ is finitely satisfiable in $A$ (using e.g. \cite{Sh715}, 1.16(1))
we can find \lseq{\a}{i}{2k} in $A$ such that the sequence
$\a_0\ldots\a_{2k-1}$ has the same $\Delta$-type as $\b_0 \ldots \b_{2k-1}$,
and so $\lseq{\a}{i}{2k}^\frown\inseq{\b}{i}{I}$ is a $\Delta$-indiscernible
set, and $k = k_\ph$ is still as in Observation \ref{obs:stabform} for this
prolonged sequence.

\begin{clm}\label{clm:2k}
Let $\c \in \fC$, then the following are equivalent
\begin{enumerate}
\item[$(a)$]
$|\set{i<2k \colon \models \ph(\a_i,\c)}|\ge k$
\item[$(b)$]
$|\set{i<2k \colon \models \ph(\b_i,\c)}|\ge k$
\item[$(c)$]
$\ph(\x,\c) \in \Av(\fU,\fC)$
\end{enumerate}
\end{clm}
\begin{prf}
  $(a) \iff (b)$ is true by the choice of $k$ and the sequence \seq{\a_i}.

  Let \lseq{\b'}{i}{\om} be a $\fU$-Shelah sequence in $p$ over $B\c$.
  Then
  $$\ph(\x,\c) \in \Av(\fU,\fC)$$
  if and only if
  $$\set{\a \in A \colon \models \ph(\a,\c)}\in \fU$$
  if and only if $$\ph(\b'_i,\c) \;\; \mbox{for all} \;\; i$$

  Note that the sequences \seq{\b_i} and \seq{b'_i} have the same type over $B$
  (by uniqueness, and even without it: both are Shelah sequences with respect to the
  same ultrafilter).



  So since $\seq{\a_i} \subseteq A$, $(a) \iff (b)$ is still true for all $\c$ if $\seq{\b_i}$
  is replaced with $\seq{\b'_i}$ (as the sequence $\lseq{\a}{i}{2k}^\frown\lseq{\b'}{i}{\om}$
  is $\Delta$-indiscernible and has the same type over $A$ as
  $\lseq{\a}{i}{2k}^\frown\lseq{\b}{i}{\om}$). Therefore,

  $$\ph(\x,\c) \in \Av(\fU,\fC)$$
  if and only if
  $$\ph(\b'_i,\c) \;\; \mbox{for all i}$$
  if and only if
  $$\ph(\a_i,\c) \;\; \mbox{for the majority of} \;\; \a_i \mbox{'s}$$

  which shows $(a) \iff (c)$.
\end{prf}

So clearly the \ph-type
$\Av_\ph(\fU,\fC)$
($\Av(\fU,\fC)$ restricted to \ph) is definable over
both the sequence $\seq{\b_0,\ldots,\b_{2k-1}}$ and the sequence
\seq{\a_0, \ldots, \a_{2k-1}}
by a \ph-formula: more specifically, $\ph(\x,\c) \in \Av(\fU,\fC)$ if and only
if
$$\bigvee_{u\subseteq 2k,|u|=k}\bigwedge_{i\in u}\ph(\b_i,\c)$$
if and only if
$$\bigvee_{u\subseteq 2k,|u|=k}\bigwedge_{i\in u}\ph(\a_i,\c)$$

This shows that $p$ is definable both over the sequence \lseq{\b}{i}{\om}
and over $A$.
\end{prf}

\begin{cor}
    A Shelah-stable type is stationary.
\end{cor}
\begin{prf}
    Note that in the proof of the theorem we showed precisely that
    the average type of a Shelah sequence in $p$ (which is an indiscernible
    set) is definable over $A$. So the conclusion follows by Proposition \ref{prp:defstat}.
\end{prf}

\begin{rmk}
    Note that this doesn't follow immediately from Corollary \ref{cor:stabstat}:
    a Shelah-stable type
    does not have to defined be over an algebraically closed set.
\end{rmk}

\begin{lem}\label{lem:MorleyShelah}
    Let $p$ be a generically stable type definable over $A$ and finitely satisfiable in
    $A$, then a Morley sequence in $p$ is a Shelah sequence in $p$
    and vice versa.
\end{lem}
\begin{prf}
    The
    desirable conclusion is an easy consequence of stationarity,  Proposition
    \ref{prp:defstat}.
\end{prf}

We conclude with a simple and natural characterization of
Shelah-stable types:

\begin{cor}\label{cor:shelahstable}
    A type $p\in\tS(A)$ is Shelah-stable if and only if it is generically stable
    and finitely satisfiable in $A$.
\end{cor}
\begin{prf}
    The ``only if'' direction is clear. So assume $p$ is finitely
    satisfiable in $A$ and generically stable; let $p'$ be an
    extension of $p$ to $\acl(A)$ finitely satisfiable in $A$.
    Since
    $\Aut(\fC/A)$ acts transitively on the set of extensions of $p$
    to $\acl(A)$ and $p$ is properly definable over $\acl(A)$, we have that
    $p'$ is definable over $\acl(A)$, and there exists a Morley sequence
    $I$ with respect to this definition which is an indiscernible set; by
    by Lemma \ref{lem:MorleyShelah} it is also a Shelah sequence,
    so we're done.
%
\end{prf}

\begin{dsc}
    So we have just shown that the two approaches to ``stable'' types
    coincide, and from now on will basically stop using the term
    "Shelah-stable", although sometimes it is convenient in order to
    indicate that the type is finitely satisfiable in its domain.
\end{dsc}

\section{Generically stable types - summary and examples}




We have chosen to define generically stable types using nonforking
indiscernible sets. There are two reasons for this. First, we find
this definition compact and elegant. The second reason is that it is
general: we do not require the sequence to be of any specific kind.
There was a price to the generality: we had to work in order to show
important properties (such as definability), which required
understanding to some extent general behavior of nonforking in
dependent theories.
But now the picture is much more complete, and we would like to
begin this section by stating several of possible alternative
definitions, some of which provide us with powerful machinery, while
others are easier to check:

\begin{thm}
Let $p \in \tS(A)$. The Following Are Equivalent:
\begin{enumerate}
\item
    $p$ is generically stable, that is, there exists a nonforking sequence in $p$ which is an indiscernible
    set.
\item
    Every nonforking sequence in $p$ is an indiscernible set.
\item
    $p$ is definable over $\acl(A)$ and some Morley sequence in $p$ is an indiscernible set.
\item
    $p$ is definable over $\acl(A)$ and every Morley sequence in $p$ is an indiscernible set.
\item
    There exists an indiscernible set $I$ in $p$ such that $\Av(I,\fC)$ does not fork over $A$.
\item
    There exists an indiscernible set $I$ in $p$ such that $\Av(I,\fC)$ does not split over $A$.
\item
    There is a nonforking extension of $p$ to a $(|A|+|T|)^+$-saturated model $M$ which is definable
    over $\acl(A)$ and finitely satisfiable in some $M_0 \prec M$ satisfying $A \subseteq M_0$ and
    $|M_0| = |A|+|T|$.
\item
    There is a nonforking extension of $p$ to a $(|A|+|T|)^+$-saturated model $M$ which is both definable
    over and finitely satisfiable in some countable indiscernible set contained in $M$.
\item
    Every nonforking extension of $p$ to a model containing $A$ is Shelah-stable.
\item
    Some nonforking extension of $p$ to a model containing $A$ is
    Shelah-stable.
\end{enumerate}
\end{thm}
\begin{prf}
    All the equivalences are easy at this point, we will sketch the proofs:

    (i) and (ii) are equivalent by stationarity and the fact that $\Aut(\fC/A)$ acts
    transitively on the set of extensions of $p$ to $\acl(A)$.

    (i) \then (iii) is basically Corollary \ref{cor:def}.

    (iii) and (iv) are again equivalent by stationarity.

    (iii) \then (i) is clear.

    (i) \then (v), (i) \then (vi) are easy by Corollary \ref{cor:def}.

    (v) \then (iii), (vi) \then (iii): Corollary \ref{cor:almostover}.

    (i) \then (vii),(i) \then (viii) are again easy.

    (vii) \then (i), (viii)\then (i) follow from Lemma \ref{lem:stabtype}
    and Lemma \ref{lem:parallel}.

    The equivalences with (ix) and (x) use stationarity, Lemma
    \ref{lem:parallel} and Corollary \ref{cor:shelahstable}.
\end{prf}

Let us now summarize our knowledge on generically stable types:

\begin{smr}
Let $p \in \tS(A)$ generically stable.
\begin{itemize}
\item
    $p$ is definable by a good definition almost over $A$.
\item
    There are boundedly many (at most $2^{|T|}$) global extensions of $p$ which do not split strongly/fork
    over $A$.
\item
    Any global extension $p'$ of $p$ which does not split strongly/fork over $A$
    is definable over $\acl(A)$. If $p$ is finitely satisfiable
    in $A$, then $p'$ is definable over $A$.
\item
    If $A = \acl(A)$ then $p$ is stationary
    and its unique global extension $p'$ which does not split strongly/fork
    over $A$ is
    a free extension with respect to some/any
    good definition over $A$.
\item
    If $p$ is finitely satisfiable in $A$ then $p$ is stationary,
    and its unique global extension which does not split strongly/fork over $A$, is
    both its coheir and a free extension with respect to some/any
    good definition over $A$.
\item
    Any nonforking extension of $p$ is generically stable. Moreover, any $q$
    which is parallel to $p$ is generically stable.
\item
    Any nonsplitting extension of $p$ to a set containing an indiscernible
    set in $p$ (in particular a slightly saturated model) is generically stable.
\item
    Any nonforking sequence in $p$ is an indiscernible set and a Morley sequence.
    Any two nonforking sequences
    in $p$ have the same type.
\item
    If $p$ is finitely satisfiable in $A$,
    any Morley sequence in $p$ is a Shelah
    sequence and vice versa.
\item
    The unique global nonsplitting extension of a stationary generically stable type $p$ is
    finitely satisfiable in and definable over any Morley (equivalently, Shelah)
    sequence in $p$.
\item
    If $A = M$ a $|T|^+$-saturated model, then $p$ is generically stable if and
    only if $p$ is both definable over and finitely satisfiable in
    some $A_0 \subseteq A$ of cardinality $|T|$. (Note that in
    general being finitely satisfiable in and definable over $M$ does
    not imply generic stability).
\end{itemize}
\end{smr}

We would like to point out particular cases which have been studied in more
detail and have become central in the recent study
of theories interpretable in o-minimal structures and in the theory of algebraically closed valued fields.
Although some of the notions below have been extensively studied by
many people over the years (and we try to mention this), we will adopt the more recent terminology,
partially due Hasson and Onshuus from \cite{HaOn2}.

We begin with the strongest version of stability which is based on the notion
of a ``stable set''. ``Stable partial types'' are originally due to Lascar and Poizat
\cite{LP}.
Let us recall the definition (since we restrain from using the term
``stable type'', we'll call this notion ``Lascar-Poizat stable''):

\begin{dfn}
    A partial type $\pi(\x)$ is called \emph{Lascar-Poizat stable} (\emph{LP-stable})
    if every extension of it to a global type is definable.
\end{dfn}

We will come back to this general concept later. The most common
terminology in case $\pi(\x)$ is finite (that is, a single formula)
is ``a stable and stably embedded set''. We give a definition which
in our opinion justifies the name ``stable'' very well:


\begin{dfn}
A definable set $D$ defined by a formula $\theta(\x)$ (maybe with parameters)
is said to be \emph{stable} if the induced structure on $D$ (including all
the relations definable on $D$ with external parameters) is stable.

\end{dfn}

We state the following fact without a proof.
Most of the equivalences are well-known. Some (which are true in any theory) were already
explored by Lascar and Poizat. Others (which require dependence) have been discovered more
recently.
All references and some proofs can be found in
Onshuus and Peterzil \cite{OnPe}.
Note that Proposition \ref{prp:herstable} below provides a generalization of
some of the following equivalences.

\begin{fct}\label{fct:stableset}
Let $D$ be a definable set defined by a formula $\theta(\x)$.
The Following Are Equivalent:
\begin{enumerate}
\item
    $D$ is stable.
\item
    $D$ is stable and stably embedded (that is, every externally definable subset of $D$
    is definable with parameters from $D$).
\item
    $D$ is Lascar-Poizat stable, that is, every global type extending $\theta(\x)$ is definable.
\item
    For every formula $\ph(\x,\y)$, the formula $\theta(\x)\land\ph(\x,\y)$ is a stable
    formula.
\item
    $D$ equipped with all the relations on it which are $D$-definable in \FC is a stable
    structure
\item
    $D$ equipped with all the externally definable relations does not have the strict
    order property. That is, there is no definable (maybe, with external parameters)
    partial order with infinite chains on $D$.
\item
    $D$ equipped with all the $D$-definable relations does not have the strict
    order property. That is, there is no definable (with no external parameters)
    partial order with infinite chains on $D$.
\end{enumerate}
\end{fct}


The following notion is due to Hasson and Onshuus, see \cite{HaOn2}.

\begin{dfn}
    A type $p \in \tS(A)$ is called \emph{seriously stable} if there exists a
    stable set $D$ defined by a formula $\theta(\x)$ (maybe with parameters)
    such that $\theta(\x) \in p$.
\end{dfn}

Obviously, this is a very strong version of stability for a type. We'll see later that
this is stronger than (and not equivalent to) the type being Lascar-Poizat stable.




In \cite{HaOn2} Onshuus and Hasson  work with the generalization
of the notion of a stable set in a dependent theory
based on Fact \ref{fct:stableset}(vi).
We'll see in Proposition \ref{prp:herstable} that just like in the case of stable sets,
this definition is equivalent to LP-stability.
The choice
of the name might seem peculiar at first, a more natural term would probably be ``$p$ does not
admit the strict order property''; it will be justified by clauses (iii) and (iv)
of  Proposition \ref{prp:herstable}.

\begin{dfn}\label{dfn:herstable}
    A type $p \in \tS(A)$ is called \emph{hereditarily stable} if there is no
    definable (maybe with external parameters)
    partial order with infinite chains on the set of realizations of $p$.
\end{dfn}

In order to show that this definition is equivalent to what one normally thinks of
as stability, we first have to recall that in a dependent theory the order property implies the strict
order property:

\begin{fct}\label{fct:sop}
    Let $\ph(\x,\y)$ be an unstable formula witnessed by indiscernible
    sequences $I = \inseq{\a}{i}{\setQ}$, $J = \inseq{\b}{i}{\setQ}$. Then
    there exists a formula $\vartheta(\y_1,\y_2,\c)$ such that
    \begin{itemize}
    \item
        $\vartheta$ defines on \FC a quasi-order
    \item
        There exists an infinite subsequence $J' \subseteq J$ which is
        linearly ordered by $\vartheta$
    \item
        $\c \subseteq \cup J$
    \end{itemize}

    \noindent
    In fact,
    $$\vartheta(\y_1,\y_2) = \forall \x [\ps(\x,\y_1) \rightarrow \ps(\x,\y_2)]$$
    for some $\ps(\x,\y,\c)$ such that
    \begin{itemize}
    \item $\ps(\x,\y,\c)$ implies $\ph(\x,\y)$
    \item $\ps$ has the strict order property
    \item $\c \subseteq \cup J$
    \end{itemize}
\end{fct}
\begin{prf}
    This is all contained in the proof of Shelah's classical theorem that in a dependent
    theory an unstable formula gives rise to the strict order property, but we would
    rather refer the
    reader to the slightly more general result by Onshuus and Peterzil, Lemma 4.1
    in \cite{OnPe}. It states that if $\ph(\x,\y)$ is unstable then there exists a
    strengthening of it (which we call here $\ps(\x,\y,\c)$)
    with the strict order property; reading the proof carefully,
    one sees both that the additional parameters are taken from $J$ and
    that the strict order property is exemplified by an indiscernible sequence
    which is an infinite subsequence of $J$. Now defining $\vartheta(\y_1,\y_2)$ as
    above, we're clearly done.
\end{prf}


We can now state the non-surprising analogue (and generalization) of Fact \ref{fct:stableset}. Some
of the equivalences below appear also in \cite{HaOn2}.

\begin{prp}\label{prp:herstable}
Let $p \in \tS(A)$. The Following Are Equivalent:
\begin{enumerate}
\item
    $p$ is LP-stable.
\item
    For every $B \supseteq A$, $p$ has at most
    $|B|^{\aleph_0}$ extensions in $\tS(B)$.
\item
    Every extension of $p$ is LP-stable.
\item
    Every extension of $p$ is generically stable.
\item
    Every indiscernible sequence in $p$ is an indiscernible set.
\item
    There is no formula $\ph(\x,\y)$ (with parameters from $p^\fC$)
    and an indiscernible sequence $\lseq{\a}{i}{\om}$
    in $p^\fC$ such that $$i<j<\om \then \ph(\a_i,\a_j)\land\neg\ph(\a_j,\a_i)$$
\item
    There is no formula $\ph(\x,\y)$ (with parameters from $\fC$)
    exemplifying the
    order property with respect to indiscernible sequences
    $I = \lseq{\a}{i}{\om}$ and $J = \lseq{\b}{i}{\om}$ with
    $\cup J \subseteq p^\fC$. We call this ``$p$ does not admit the
    order property''.
\item
    $p$ is hereditarily stable as in Definition \ref{dfn:herstable}; that is,
    $p$ does not admit the strict order property.
\item
    On the set of realizations of $p$
    there is no definable (with no external parameters) partial order with
    infinite chains.
\end{enumerate}
\end{prp}
\begin{prf}

    The equivalence of (i) and (ii) is
    well-known (Theorem 4.4 in \cite{LP}).

     (i) \iff (iii) is trivial.

    (ii) \then (vii):
    assume $p$ admits
    the order property. Using the standard argument, for every infinite $\lam \ge |A|$
    one can easily construct a collection of $\lam^+$ extensions of $p$ over a set
    of cardinality $\le \lam$; clearly, this contradicts (ii).

    (vii) \then (vi): Clear.

    (vi) \then (v) is standard: e.g., taking an indiscernible sequence in
    \lseq{\a}{i}{\om+\om} in $p$ which is not an indiscernible set, we may assume
    that for some formula $\ph(\z,\x,\y,\z')$ and $n<\om$ we have
    $\ph(\a,\a_n,\a_{n+1},\a')$ and $\neg\ph(\a,\a_{n+1},\a_n,\a')$ where
    $\a = \a_{<n}$, $\a' \subseteq \cup\a_{>\om}$. Now adding $\a\a'$ to the parameters,
    we obtain the sequence $\seq{\a_i\colon n\le i<\om}$ as required.

    (v) \then (iv): Clear.

   (iv) \then (i): Clear.

    \medskip

    So we showed (i) \iff (ii) \iff (iii) \then (vii) \then (vi) \then (v) \then (iv) \then (i). This
    completes all the equivalences except (viii) and (ix).

    \medskip

    (vii) \then (viii), (viii) \then (ix) are trivial.

    (ix) \then (vii) Let $\ph(\x,\y), I, J$ be as in $\neg$(vii),
    without loss of generality both $I$ and $J$ are of order type $\setQ$. By Fact \ref{fct:sop} there
    exists
    $\vartheta(\x,\y)$ (maybe with additional parameters from $\cup J \subseteq p^\fC$)
    which defines a partial order on $\fC$ and linearly orders an infinite subsequence of
    $J$ which lies in $p^\fC$; so we're done.


\end{prf}






\begin{rmk} A curious point: since we do not use clauses (viii) and (ix) of Proposition
\ref{prp:herstable} in the proof of the equivalence of (i) -- (vii), we also obtain
an alternative proof of Proposition 4.2 in \cite{OnPe}
(weak stability implies stability, even for a type) that goes through generically
stable types.
\end{rmk}

We now proceed to the third version of stability which is due to Haskell, Hrushovski and
Macpherson and is studied in great detail in \cite{HHM}. We give an equivalent definition which appears
in Hrushovski \cite{Hr}.

\begin{dfn}
    A type $p \in \tS(A)$ is called \emph{stably dominated} if
    there exists a collection of stable sets $\bar D = \seq{D_i\colon i<\al}$ and definable functions
    $f_i \colon p^\fC \to D_i$ such that for every set $B\supseteq A$ and
    $\a \models p$, if $f_i(a) \ind^{st}_A B$ for all $i$ (which in this context
    just means that $\tp(f_i(a)/B)$ is definable over $A$),
    then (denoting $\bar f = \lseq{f}{i}{\al}$) $\tp(B/A\bar f(\a)) \vdash \tp(B/A\a)$.

    In this case we also say that $p$ is stably dominated \emph{by $\bar D$ via $\bar f$}.
\end{dfn}

\begin{obs}
\begin{enumerate}
\item
    A seriously stable type is stably dominated.
\item
    A seriously stable type is hereditarily stable.
\item
    A stably dominated type is generically stable.
\item
    A hereditarily stable type is generically stable.
\end{enumerate}
\end{obs}
\begin{prf}
    The only nontrivial statement here is (iii); but it is easy to
    deduce using properties of independence of stably dominated
    types, see e.g. Proposition 2.8 in \cite{Hr}, that a Morley sequence
    in a stably dominated type is an indiscernible set.
\end{prf}

The following examples show that the notions ``hereditarily stable'' and
``stably dominated'' are ``orthogonal'', that is, none of them implies
the other. Both of these examples were used by Hasson and Onshuus in \cite{HaOn1}
for different purposes.

\begin{exm}\label{exm:intervals}
    Let us consider the theory of $\setQ$ with a predicate $P_n$ for every
    interval $[n,n+1)$ ($n\in \setZ$) and the natural order $<_n$ on $P_n$.
    It is easy to see that the ``generic'' type ``at infinity'' (that is, the type of
    an element not in any of the $P_n$'s) is hereditarily stable. It is not
    stably dominated since there are no stable sets. In particular, it is
    not seriously stable.

    Note that this theory is interpretable in the o-minimal
    theory $(\setQ,+,<)$ and therefore dependent.
\end{exm}

\begin{exm}
    Let us consider the theory of a two-sorted structure $(X,Y)$: on $X$ there is
    an equivalence relation $E(x_1,x_2)$
    with infinitely many infinite classes and each class densely linearly
    ordered, while $Y$ is just an infinite set such that there is a definable function
    $f$ from $X$ onto $Y$ with $f(a_1) = f(a_2) \iff E(a_1,a_2)$.

    In other words, $Y$ is the sort of imaginary elements corresponding to the classes of $E$.
    Clearly $Y$ is stable and stably embedded.

    Let $M$ a model and $p$ the ``generic'' type
    in $X$ over $M$, that is, a type of an element in a new equivalence class.
    Pick $a \models p$ and $B \supseteq M$ such that $a \ind^{st}_M B$, that is,
    $\tp(a/B)$ is definable over $M$, which necessarily means $\tp(a/B)$ is generic
    in the sense above, that is, $B$ does not contain any elements of the equivalence class
    of $a$. So clearly $\tp(B/Ma)$ is completely determined by $\tp(B/Mf(a))$.

    This shows that $p$ is stably dominated via $f$ and $Y$. It is clearly not
    hereditarily stable (e.g. admits the strict order property).
\end{exm}

    We will give now several
    examples of generically stable types which are not hereditarily stable
    or stably dominated.

The following example is basically due to Kobi Peterzil. A version
of it discussed in more detail by Hasson and Onshuus in
\cite{HaOn2}.

\begin{exm}\label{exm:FDO}
    Let $FDO$ (``$FDO$'' stands for ``Finite Dense Orders'')
    be the theory of $\setQ$ equipped with predicate symbols
    $<_n$ for $n \in \setN$ such that $<_n$ defines an order on
    rational numbers of distance at most $n$. That is, $\models
    q_1<_nq_2$ if and only if $|q_1-q_2|\le n$ and $q_1<q_2$. Clearly
    $q_1 <_n q_2$ implies $q_1 <_m q_2$ for all $m>n$.

    Let $M \models FDO$ and $p$ be the ``infinity'' type, that is, the
    type of an element which is not comparable to any element of $M$ with
    respect to any of the finite orders. Clearly $p$ is generically stable and a Shelah/Morley
    sequence in $p$ is just a set of pairwise incomparable elements which are "infinitely
    far" from each other.
    On the other hand, there exists an indiscernible sequence in $p$ which is
    increasing with respect to (for example) $<_{45}$ but not $<_{44}$.

    So there are
    many different extensions of $p$ which are not generically stable, hence $p$ is
    not hereditarily stable. Just like in Example \ref{exm:intervals}, there are no
    stable sets in $FDO$ and therefore no stably dominated types.

    Note that a similar phenomenon can be obtained by starting with the theory from Example \ref{exm:intervals}
    expanded with the group structure on $\setQ$ and taking a reduct. Just like the theory
    in \ref{exm:intervals}, $FDO$ is interpretable in $(\setQ,+,<)$.
\end{exm}

The second example arises in a more natural context:

\begin{exm}\label{exm:RV}
    Let $RV$ be a two-sorted theory of a real closed (ordered) field $R$ and an infinite
    dimensional vector space $V$ over it. There is a definable partial order on $V$:
    $$v_1\le v_2\iff \exists r\in R,r\ge 1_R \;\;\mbox{such that}\;\;v_2 = r\cdot v_1$$

    Let $M$ be a model and $p \in \tS(M)$ be the type of a generic vector. Then $p$ is generically stable
    and every Morley/Shelah sequence is an indiscernible linearly independent set. On the other
    hand, there are (for example) increasing indiscernible sequences in $p$, so $p$ is not hereditarily
    stable. Like in the previous examples, there are no stable sets, and therefore no stably
    dominated types.
\end{exm}

Note that one could define a more general notion of stable domination, using a hereditarily stable type
instead of a collection of stable sets, as is done in \cite{HP}:

\begin{dfn}\label{dfn:herdom}
    We call $p \in \tS(A)$ is called \emph{stably dominated} if
    there exists a collection of LP-stable partial types $\bar \pi = \seq{\pi_i\colon i<\al}$ and definable functions
    $f_i \colon p^\fC \to \pi_i$ such that for every set $B\supseteq A$ and
    $\a \models p$, if $f_i(a) \ind^{st}_A B$ for all $i$ (which
    just means that $\tp(f_i(a)/B)$ is definable over $A$),
    then (denoting $\bar f = \lseq{f}{i}{\al}$) $\tp(B/A\bar f(\a)) \vdash \tp(B/A\a)$.
\end{dfn}

Clearly, working with this definition, every hereditarily stable type is stably
dominated. Still, generically stable types given
in Examples \ref{exm:FDO} and \ref{exm:RV} are not stably dominated even in this stronger sense (there are no
hereditarily stable types).

\begin{dsc}\label{dsc:FDO}
    We would like to point out a phenomenon which can be seen in both examples \ref{exm:FDO} and
    \ref{exm:RV}.
    Let us consider e.g. $T = FDO$. Let $M$ be a model, $p \in \tS(M)$ the generically stable
    type ``at infinity'', $I = \lseq{b}{i}{\om}$ a Morley sequence in $p$.

    Let $J = \lseq{b'}{i}{\om}$ be a $<_1$-increasing sequence in $p$ with $b'_0 = b_0$. Note
    that $q=\tp(b'_1/Mb_0)$ is not generically stable and does not split over $M$. This shows that nonsplitting
    extensions of generically stable types over arbitrary sets are not necessarily generically stable, even if the domain
    of the original type is a model (making it saturated wouldn't help).

    Clearly $q$ does not have an extension
    to $B = M\cup I$ which doesn't split over $M$; otherwise, $I$ would be indiscernible
    over $Mb'_1$, that is, $b'_1$ would be $<_1$-bigger than all elements of $I$, which is
    absurd since elements of $I$ are $<_n$ incomparable for all $n$. Of course we know
    another reason that suggests that such an extension doesn't exist: any such extension
    must be generically stable, and as a matter of fact, it is unique and equals to $\Av(I,M\cup I)$.

    Moreover, note that $J$ is a nonsplitting sequence
    in $p$ over $M$. Obviously it is not a Morley or a Shelah sequence. This shows that
    it is not true that a generically stable type has a unique nonsplitting sequence, or even that
    every nonsplitting sequence in a generically stable type must be an indiscernible set.
\end{dsc}

Generically stable types which are not hereditarily stable or stably
dominated are generally difficult to handle because they do not have
to be at all related to the ``stable'' part of the theory; in fact,
the theory does not even have to have any ``stable'' part, like in
Examples \ref{exm:FDO} and \ref{exm:RV}. Still, our results apply in
the most general case. The next section generalizes the independence
relation developed for stably dominated types in \cite{HHM} to
generically stable types.

\section{Generically stable types and forking independence}
Let $p \in \tS(A)$ be a generically stable type. In particular
it is properly definable over $\acl(A)$ by a definition schema
$d_p$.
We will denote the free extension of $p$ to $\fC$
by $p |^d \fC$, or $p|\fC$ when $d$ is clear from the context. For a set $B$
let $p | B = p | \fC \rest B$.

\begin{dfn}
For $\a \models p$ we say that it is \emph{stably forking
independent} (or just \emph{forking independent}) of $B$ over $A$,
$\a\ind_A B$, if $\tp(\a/AB)$ does not fork over $A$.
\end{dfn}

\begin{obs}\label{obs:inddef} Let $\tp(\a/A)$ be generically stable, then $\a\ind_AB$  if and only if $a\models p|^d B$
with respect to one of the (boundedly many) definitions of $p$ over
$\acl(A)$.
\end{obs}
\begin{prf}
    By Observation \ref{obs:nonfork} and stationarity.
\end{prf}

\begin{rmk}
    Observation \ref{obs:inddef} implies in particular that forking independence generalizes
    independence relation developed for stably dominated types in \cite{HHM}.
\end{rmk}

\begin{lem}\label{lem:symm}
(Symmetry Lemma) Let $p,q\in\tS(A)$ be generically stable types, $\a\models p$,
$\b \models q$. Then $\a\ind_A\b \iff \b\ind_A\a$.
\end{lem}
\begin{prf}
  Suppose not. Assume for example that $\a\ind_A\b$ and $\b\nind_A\a$. By Observation \ref{obs:inddef} there is
  $\ph(\x,\y)$ such that $\ph(\a,\b)$, $\models d_p\x\ph(x,\b)$, but
  $\models\neg d_q\y\ph(\a,\y)$. Let $\a_0 = \a$, $\b_0 = \b$.
  Now construct sequences $\seq{\a_i}$,
  \seq{\b_i} for $i<\om+\om$ as follows:

  $$ \a_i \models p | A\lseq{\a}{j}{i}\lseq{\b}{j}{i} $$
  $$ \b_i \models q | A\lseq{\a}{j}{i+1}\lseq{\b}{j}{i} $$

  These sequences exemplify the order property for $\ph(\x,\y)$.
  But they are indiscernible sets (as $p$ and
  $q$ are generically stable), so $\ph(\x,\y)$ is supposed to be stable with respect
  to them, see Observation \ref{obs:stabform}, and this is clearly not the case, take
  e.g. $\ph(\a_\om,\b_i)$ which holds for $i<\om$ and fails for $i>\om$.
\end{prf}

\begin{rmk}\label{rmk:symdef} Note that for the proof of Symmetry Lemma we only need one of the types to
be generically stable, and another one to only be definable (since
it is enough to get a contradiction to stability of $\ph$ with
respect to one of the sequences). In fact, a slight modification of
the proof shows that even definability is not necessary; see Lemma
\ref{lem:strongsymm} for a strong symmetry result.
\end{rmk}







\begin{thm}
    Let $p,q \in S(A)$ be generically stable, $\a,\b$ realize $p,q$ respectively, and
    let $\c, \d$ be any tuples (maybe infinite). Then:
    \begin{itemize}
    \item \emph{Irreflexivity} $\a \ind_A \a$ if and only if $p$ is algebraic
    \item \emph{Monotonicity} If $a \ind_A \b\c\d$, then $a \ind_A\c \b$.
    \item \emph{Symmetry} $\a \ind_A \b$ if and only if $\b \ind_A \a$
    \item \emph{Transitivity} $\a \ind_A \c\d$ if and only if
    $\a \ind_{A\c}\d$ and $\a \ind_A\c$
    \item \emph{Existence} Let $B \supseteq A$, then there exists
    $\a' \equiv_A \a$ such that $\tp(\a'/B)$ is generically stable and
    $\a'\ind_AB$.
    \item \emph{Uniqueness} If $\a \ind_A \c$ , $\a' \ind_A \c$ and $\a' \equiv_{\acl(A)}\a$,
    then $\a \equiv_{A\c} \a'$
    \item \emph{Local Character} If $\a \ind_A \c$, then for some
    subset $A_0$ of $A$ of cardinality $|T|$, $\a \ind_{A_0} \c$.
    If $A = M$ is an $\aleph_1$-saturated model, there exists a countable $A_0
    \subseteq M$ such that $\a \ind_{A_0} \b$.
    \end{itemize}
\end{thm}

\begin{prf}
    By the definitions and previous results: e.g.,
    Transitivity is clear from Observation \ref{obs:inddef},
    Existence is just existence of nonforking extensions, for Uniqueness
    use stationarity of generically stable types, for Local Character take
    $A_0$ to be a Morley sequence in $p$ over $\b\a$.









\end{prf}

As usual, we will call a set $\cB$ of tuples realizing generically stable types
\emph{forking independent} over a set $A$ if for every $\cB_0
\subseteq \cB$ we have $\cup\cB \ind_A \cup(\cB\setminus\cB_0)$. We
call a sequence $I = \inseq{\b}{i}{O}$ of realizations of generically stable
types \emph{forking independent} if the set $\set{\b_i\colon i\in
O}$ is forking independent. Just like in stable theories, using the
properties of stable forking independence, $I$ is forking
independent if and only if for every $i \in O$ we have $\b_i \ind_A
\b_{<i}$.

Let $A = \acl(A)$, $p,q \in \tS(A)$ generically stable types. We denote by $p
\otimes q$ the unique (by stationarity) type of an independent (over
$A$) pair $(\a,\b)$ of realizations of $p$ and $q$ respectively. If
$p = q$ we also write $p^{\otimes 2}$ for $p\otimes p$, and
generally denote $p^{\otimes n} = p \otimes \ldots \otimes p$ $n$
times, which is well-defined by the properties of forking
independence.

It is easy to see that

\begin{obs}\label{obs:pairstable}
    Given $p,q \in \tS(A)$ generically stable, $p\otimes q \in
    \tS(A)$ is generically stable. Similarly for a product of any
    number (finite or infinite) of generically stable types.
\end{obs}

Like in stable theories, we have

\begin{rmk}\label{rmk:indind}
    Let $I = \inseq{\b}{i}{O}$ be a sequence forking independent
    over a set $A = \acl(A)$ of realizations of the same generically stable type
    $p$ over $A$. Then $I$ is an indiscernible set over $A$, the type of an
    $n$-tuple $\b_{i_1}\ldots\b_{i_n}$ being $p^{\otimes n}$.
\end{rmk}

One can characterize generically stable types in terms of the
properties of forking on the set of their realizations. Of course
there are many different such characterizations, we start with the
simplest ones, which also
come handy in Lemma \ref{lem:gendom}:

\begin{obs}\label{obs:symunique}
    Let $p \in \tS(A)$ be a type. Then the following are equivalent:
    \begin{enumerate}
    \item
        $p$ is generically stable
    \item
        For every $\a, \a' \models p$ and $B,C,D \subseteq p^\fC$ we have ($B\ind_AC$ stands
        for ``$\tp(B/AC)$ does not fork over $A$''):
        \begin{itemize}
        \item
            Symmetry: $B\ind_A C \iff C\ind_AB$
        \item
            Transitivity: $B \ind_A C$ and $B \ind_{AC}D
            \then B \ind_A CD$
        \item
            Uniqueness: $\a\ind_AB, \a'\ind_AB \then \a\equiv_{AB}\a'$
        \end{itemize}
    \item
        $p$ is properly definable over $\acl(A)$ by a definition scheme $d$ and
        for every $\a_1,\ldots,\a_n \models p$ we have:
        \begin{itemize}
        \item
            Symmetry: if $\a_i \models p|^d A\a_{<i}$ for all $i$
            then for every permutation $\sigma \in S_n$ we have
            $\a_{\sigma(i)}\models p|^dA\a_{\sigma(<i)}$ for all
            $i$.
        \end{itemize}

    \end{enumerate}
\end{obs}
\begin{prf}

    (ii)\then (i): one chooses a
    nonforking sequence in $p$ and shows using symmetry, transitivity and
    uniqueness that it is an indiscernible set.

    (iii)\then (i) is even easier: for any definable type, a Morley sequence is
    indiscernible. Symmetry does the rest.






\end{prf}






Following Definition \ref{dfn:herdom}, one could try to generalize
stable domination to generically stable types. The following lemma
shows that this does not lead to anything new, which confirms our
perception of generic stability as the most general notion of
stability for a type.

\begin{lem}\label{lem:gendom}
    Let $q \in \tS(A)$ be a generically stable type.
    Assume that $p \in \tS(A)$ is \emph{stably dominated by $q$ via
    $f$}, that is, assume that $f$ is a definable function from $p^\fC$ to $q^\fC$ such that
     for every set $B\supseteq A$ and
    $\a \models p$, if $f(a) \ind^{st}_A B$
    then $\tp(B/Af(\a)) \vdash \tp(B/A\a)$.

    Then $p$ is generically stable.
\end{lem}
\begin{prf}
    Using Lemma 3.12 in \cite{HHM}, it is easy to see that $p$ is properly definable over $\acl(A)$.
    Moreover, ``pulling back'' to $p$ via $f$ properties of stable forking independence on $q$, one shows
    that definable extensions satisfy Symmetry as in Observation \ref{obs:symunique}(iii).
    Hence $p$ is generically stable.

\end{prf}

\begin{rmk}
    In \cite{OnUs2} Onshuus and the author provide a generalization of this
    Lemma, replacing stable domination with forking domination.
\end{rmk}

It would be interesting to investigate properties mentioned in
Observation \ref{obs:symunique} on their own: which ones imply each
other, which imply generic stability, etc. We do not pursue this
direction much further here and only make a few remarks.

\begin{obs}\label{obs:statstable}
    Let $A = \acl(A)$, $p \in \tS(A)$ be definable. Then $p$ is
    generically stable if and only if $p$ is stationary.
\end{obs}
\begin{prf}
    Taking a free extension we may assume $A$ is a model. Now every
    nonforking sequence in $p$ is by stationarity both a Morley and
    a  Shelah sequence. It is easy to see that this implies generic
    stability.
\end{prf}

\begin{lem}\label{lem:uniquestat}
\begin{enumerate}
\item
    Let $A$ be a model (or just $A = \bdd^{heq}(A)$), $p \in \tS(A)$ be a type satisfying Uniqueness of nonforking
    as in Observation \ref{obs:symunique}(ii). Then $p$ is
    stationary.
\item
    Let $p$ be a Lascar strong type satisfying Uniqueness of nonforking
    as in Observation \ref{obs:symunique}(ii) with ``type'' replaced by
    ``Lascar strong type''. Then $p$ is
    stationary.
\item
    Let $p \in \tS(A)$ be a definable type which satisfies Uniqueness of
    definable extensions in the following sense: any two
    definitions of it over $A$ agree on the set of realizations of
    $p$. Then $p$ has a unique global extension definable over $A$.
\end{enumerate}
\end{lem}
\begin{prf}
\begin{enumerate}
\item
    Assume that $p$ has two global nonforking extensions $q_1$ and
    $q_2$. Then there is $\c\in\fC$ and a formula $\ph(\x,\y)$ such
    that $\ph(\x,\c) \in q_1$, $\neg\ph(\x,\c) \in q_2$. Now
    define a sequence of realizations of $p$, $I = \lseq{\a}{i}{\om}$
    as follows:
    $$\a_{2i} \models q_1\rest A\c\a_{<2i}$$
    $$\a_{2i+1} \models q_2\rest A\c\a_{<2i+1}$$
    This is a nonforking sequence. Since $A$ is a model,
    it is also nonsplitting. By Uniqueness of
    nonforking extensions on the set of realizations of $p$ (that
    is, by the assumption) and Fact \ref{fct:splitind}, it is
    indiscernible. Now $\ph(\x,\c)$ and $I$ clearly contradict
    dependence.
\item
    Same proof with splitting replaced by Lascar splitting (using
    Observation \ref{obs:lascarsplit}) and Fact \ref{fct:splitind}
    replaced with Observation \ref{obs:splitind}.
\item
    Similar.
\end{enumerate}
\end{prf}

\begin{cor}
\begin{enumerate}
\item
    Let $A$ be a model (or just $A = \bdd^{heq}(A)$), $p \in \tS(A)$ be a definable type which satisfies
    Uniqueness of nonforking as in Observation
    \ref{obs:symunique}(ii). Then $p$ is generically stable.
\item
    Let $p$ be a definable Lascar strong type over $A$ which satisfies
    Uniqueness of nonforking as in Observation
    \ref{obs:symunique}(ii) with ``type'' replaced with ``Lascar strong type''. Then $p$ is generically stable.
\end{enumerate}
\end{cor}
\begin{prf}
    By Lemma \ref{lem:uniquestat} and Observation \ref{obs:statstable}.
\end{prf}

Note that uniqueness of definable extensions does not imply generic
stability: the type ``at infinity'' in the theory of a dense linear
order $(\setQ,<)$ is definable and has a unique global definable
extension. It is not, of course, stationary or generically stable.






\section{Strong stability and stable weight}
In this section we develop the basic theory of stable weight of a
type, that is, weight with respect to generically stable types. Our
hope is that in a dependent theory it is possible to ``analyze'' an
arbitrary type with respect to its ``stable-like'' part and a
``partial order''. The goal of stable weight is to provide certain
understanding of the ``stable'' part.

We aim to connect finiteness of stable weight to strong dependence
introduced by Shelah in \cite{Sh783} and studied more intensively in
\cite{Sh863}.  The following definitions are motivated by those
notions.

\begin{dfn}
\begin{enumerate}
\item
    A \emph{randomness pattern} of depth $\ka$ for a (partial) type $p$ over a set $A$ is an
    array $\seq{\b_i^\al \colon \al<\ka, i<\om}$ and formulae $\ph_\al(\x,\y_\al)$ for
    $\al<\ka$ such that
    \begin{enumerate}
    \item
        The sequences $J_\al = \seq{\b^\al_i\colon i<\om}$ are
        mutually indiscernible over $A$; more precisely, $J_\al$ is
        indiscernible over $AJ_{\neq \al}$
    \item
        $\len(\b^\al_i) = \len(\y_\al)$
    \item
        for every $\eta \in {}^\ka\om$, the set
        $$ \Gamma_\eta = \set{\ph_\al(\x,\b^\al_\eta(\al) \colon \al < \ka} \cup
        \set{\neg\ph_\al(\x,\b^\al_i) \colon \al<\ka, i<\om, i\neq \eta(\al)}$$
        is consistent with $p$.
    \end{enumerate}
\item
    A (partial) type $p$ over a set $A$ is called \emph{strongly dependent} if there do not exist
    formulae $\ph_\al(\x,\y_\al)$ for $\al<\om$ and
    sequences \seq{\b^\al_i \colon i<\om} for $\al<\om$
    mutually indiscernible over $A$ such that
    for every $\eta \in {}^\om\om$, the set
    $$ \Gamma_\eta = \set{\ph_\al(\x,\b^\al_\eta(\al) \colon \al < \om} \cup
    \set{\neg\ph_\al(\x,\b^\al_i) \colon \al<\om, i\neq \eta(\al)}$$
    is consistent with $p$.

    In other words, $p$ is called strongly dependent if there does not exist a randomness
    pattern for $p$ of depth $\ka = \om$.
\item
    \emph{Dependence rank} (dp-rk) of a (partial) type $p$ over a set $A$
    is the supremum of all $\ka$ such that there exists a randomness pattern
    for $p$ of depth $\ka$.
    \item
    A (partial) type over a set $A$ is called \emph{dp-minimal} if
    dp-rank of $p$ is 1.

    In other words, $p$ is dp-minimal if there does not exist a randomness pattern
    for $p$ of depth $2$.
\item
    A theory is called strongly dependent/dp-minimal if the partial
    type $x=x$ is.
\item
    Let $T$ be dependent. A type $p$ is called \emph{strongly stable} if it is strongly dependent and
    generically stable.
\end{enumerate}
\end{dfn}

\begin{rmk}
    For a partial type $p$, $\dpr(p)\ge 1$ iff $p$ is
    nonalgebraic.
\end{rmk}
\begin{prf}
    The ``only if'' direction is obvious. For the ``if'' direction,
    by non-algebraicity, the formula $x=y$ does the trick.
\end{prf}

\begin{rmk}
    A very close relative of $\dpr$ is called ``burden'' by Hans Adler in
    \cite{Ad2}. He also studies ``strong'' theories which is a class
    containing strongly dependent theories, but also some
    independent ones, e.g. supersimple theories, and more.
\end{rmk}

We can define the \emph{stable weight} of $p$, $\swt(p)$ as weight
of $p$ with respect to generically stable types:

\begin{dfn}
\begin{enumerate}
\item
    Let $p \in \tS(A)$ be a type. We define the \emph{stable pre-weight} of $p$,
    $\spwt(p)$, to be the supremum of all $\al$ such that
    there exist $\a \models p$, generically stable types $\lseq{q}{i}{\al}$ over $A$ and
    $\b_i \models q_i$ such that:
    \begin{itemize}
    \item
        $\set{\b_i \colon i<\al}$ is an independent
        set over $A$
    \item
        $\tp(\a/A\b_i)$ divides over $A$ for all $i$
    \end{itemize}
\item
    The \emph{stable weight} of $p$, $\swt(p)$ is the supremum of
    the stable pre-weights of all nonforking extensions of $p$.
\end{enumerate}
\end{dfn}

The main goal of this section is to show that a strongly dependent
type has finite stable weight. We will need the following slightly
surprising strengthening of the Symmetry Lemma. Recall that Remark
\ref{rmk:symdef} states that for the proof of Lemma \ref{lem:symm}
it is enough to assume that one of the types is generically stable
and the other one is definable. We intend to eliminate definability
from the assumptions.

For simplicity of notation, we will denote ``$\tp(B/AC)$ does not
fork over $A$'' by ``$B\ind_AC$'' even if $\tp(B/A)$ is not
generically stable (and so the relation above does not need to be
symmetric).

\begin{lem}\label{lem:strongsymm}
(Strong Symmetry Lemma) Let $p\in\tS(A)$ be generically stable, $q
\in \tS(A)$ does not fork over $A$, $\a\models p$, $\b \models q$.
Then
\begin{enumerate}
\item
$\a\ind_A\b \Longrightarrow \b\ind_A\a$. Moreover, if $A = \acl(A)$
and $\a\ind_A\b$, then there exists a unique nonforking extension of
$q$ to $\tS(A\a)$ which equals $\tp(\b/A\a)$.
\item
$\b\ind_A\a \Longrightarrow
\a\ind_A\b$.
\end{enumerate}
\end{lem}
\begin{prf}
\begin{enumerate}
\item
  Clearly, it is enough to prove the lemma for
  $A = \acl(A)$. Let $q^*$ be a global nonforking extension of $q$.
  We will show
  that $q^*\rest A\a = \tp(\b/A\a)$, proving the moreover part as
  well.

  Suppose not.
  Then there is a formula $\ph(\x,\y)$ such that $\ph(\a,\b)$ (so $d_p\x\ph(\x,\b)$ holds), but
    $\neg\ph(\a,\y) \in q^*$.

  Let $\a_0 = \a$, $\b_0 = \b$.
  Construct sequences $\seq{\a_i}$,
  \seq{\b_i} for $i<\om+\om$ as follows:

  $$ \a_i \models p | A\lseq{\a}{j}{i}\lseq{\b}{j}{i} $$
  $$ \b_i \models q^* \rest A\lseq{\a}{j}{i+1}\lseq{\b}{j}{i} $$

  Now note:

  \begin{itemize}
  \item
    $j<i \then \ph(\a_i,\b_j)$: since $\models \d_p\x\ph(\x,\b)$,
    $\b \equiv_A \b_j$ and $\a_i$ is chosen generically over $A\b_j$
  \item
    $j\ge i \then \neg\ph(\a_i,\b_j)$: since $\neg\ph(\a,\y)\in q^*$,
    $q^*$ does not fork hence does not Lascar split over $A$,
    $\a \equiv_{Lstp,A} \a_i$ (in fact, they are of Lascar distance
    1) and $\b_j$ was chosen to realize $q^*$ over $A\a_i$.
  \end{itemize}

  As in the proof of Lemma \ref{lem:symm}, this is a contradiction
  to generic stability of $p$ (that is, $\lseq{\a}{i}{\om+\om}$ being an
  indiscernible set).

\item
    Now assume $\a\nind_A\b$, so there is a formula $\ph(\x,\y)$
    such that $d_p\x\ph(\x,\b)$ but $\neg\ph(\a,\b)$ holds. Again without loss
    of generality $A = \acl(A)$.


    Let $\a'
    \models p$ such that $\a' \ind_A \b$, so $\ph(\a',\b)$. By
    clause (i) of the Lemma we have
    \begin{equ}\label{equ:strongsymm1}
    The only nonforking extension of $q$ to $A\a'$ is
    $\tp(\b/A\a')$.
    \end{equ}

    On the other hand, $\neg\ph(\a',\y)$ is consistent with $q$ (by
    applying an automorphism of $\tp(\b/A\a)$ over $A$ taking $\a$ to
    $\a'$). So $q\cup\set{\neg\ph(\a',\y)}$ forks over $A$ (by
    \ref{equ:strongsymm1}), hence so does $q \cup
    \set{\neg\ph(\a,\y)}$, which is a subset of $\tp(\b/A\a)$, as
    required.


\end{enumerate}

\end{prf}

We will make use of the following well-known fact (due to Morley):

\begin{fct}\label{fct:Erdos Rado}
    Let $\lam$ be a cardinal. Then there exists $\mu>\lam$ such that
    for every set $A$ of cardinality $\lam$ and a sequence of tuples
    \lseq{a}{i}{\mu} there exists an \om-type $q(x_0,x_1,\cdots)$ of
    an $A$-indiscernible sequence such that for every $n<\om$ there
    exist $i_1<i_2<\ldots<i_{n}<\mu$ such that
    the restriction of $q$ to the first $n$ variables equals
    $\tp(a_{i_1}\ldots a_{i_n}/A)$.

    We will sometimes denote $\mu$ as above by $\mu(\lam)$.
\end{fct}

We will also need some basic facts about nonforking calculus and
preservation of independence in dependent theories. Recall that
``$B\ind_AC$'' stands for ``$\tp(B/AC)$ does not fork over $A$''.

\begin{fct}\label{fct:lefttrans}(Shelah)
Let $A, B$ be sets and assume that $I = \lseq{a}{i}{\lam}$ is a
nonforking sequence based on A, that is ,$a_i \ind_ABa_{<i}$ for all
$i<\lam$. Then $I\ind_A B$, that is, $\tp(I/AB)$ does not fork over
$A$.
\end{fct}
\begin{prf}
This is Claim 5.16 in \cite{Sh783}.
\end{prf}

\begin{cor}\label{cor:indep set}
    Let $\set{A_i\colon i<\lam}$ be a nonforking (independent) set over $A$,
    that is, $A_i \ind_A A_{\neq i}$ for all $i$. Then for every
    $W,U \subseteq \lam$ disjoint we have $A_{\in W}\ind_AA_{\in
    U}$.
\end{cor}
\begin{prf}
    Monotonicity and transitivity on the left.
\end{prf}

\begin{obs}\label{obs:indispreserve}
    Suppose $I$ is an indiscernible sequence over $A$ and
    $B \ind_A I$. Then $I$ is indiscernible over $AB$.
\end{obs}
\begin{prf}
    By Fact \ref{fct:splitfork}  $\tp(B/AI)$ does not split strongly
    over $A$. Recall that this implies that for every $\a_1,\a_2 \in
    I$ which are on the same $A$-indiscernible sequence we have
    $B\a_1 \equiv_A B\a_2$, which is precisely what we want.
\end{prf}

The following lemma is the key to the proof of the main theorem. It
shows that indiscernible sequences which start with generically
stable independent elements can be assumed to be mutually
indiscernible.

\begin{lem}\label{lem:mutindisc}
    Let $A$ be an extension base (that is, no
    type over $A$ forks over $A$; e.g. $A$ is a model).
    Let $\set{\a_i \colon i<\al}$ be an $A$-independent set of
    elements satisfying generically stable types over $A$, and let
    $\lseq{I}{i}{\al}$ be a sequence of $A$-indiscernible
    sequences starting with $\a_i$ respectively.  Then there exist sequences
    \lseq{I'}{i}{\al} such that
    \begin{itemize}
    \item
        $I'_{i} \equiv_{A} I_i$
    \item
        $I'_i$ starts with $\a_i$
    \item
        $I'_{i}$ is indiscernible over $AI'_{\neq i}$
    \end{itemize}
\end{lem}
\begin{prf}
    By compactness it is enough to take care of $\al = k$ finite. So
    we will prove the lemma by induction on $k$, the case $k=1$
    being trivial. In fact, we will prove more, that is, we will
    prove by induction on $k$ that there there are \lseq{I'}{i}{k} such that
    for all $i<k$
    \begin{equ}\label{equ:mutindisc1}
    \item
    \begin{itemize}
    \item
        $I'_{i} \equiv_{A} I_i$
    \item
        $I'_i$ starts with $\a_i$
    \item
        $I'_{i}$ is indiscernible over $AI'_{\neq i}$
    \item
        $I'_{i} \ind_A I'_{< i}$
    \end{itemize}
    \end{equ}

    Let $\set{\a_i\colon i<k+1}$, \lseq{I}{i}{k+1} be as in the assumptions of the Lemma. By
    the induction hypothesis we may assume that \lseq{I}{i}{k} are
    as in \ref{equ:mutindisc1} above.



    Since $\a_{<k} \ind_A \a_k$, by Fact \ref{fct:nonforkexist} there exist $\a'_{<k}
    \equiv_{A\a_k} \a_{<k}$ such that $\a'_{<k} \ind_A I_k$. By
    applying an automorphism over $A\a_k$, we may assume $\a'_{<k} =
    \a_{<k}$. More specifically, let $\sigma \in \Aut(\fC/A\a_k)$
    take $\a'_{<k}$ to $\a_{<k}$. Denote $I''_k = \sigma(I_k)$. Then

    \begin{itemize}
    \item
        $I''_{k} \equiv_{A} I_k$
    \item
        $I''_k$ starts with $\a_k$ ($\sigma$ does not move $\a_k$)
    \end{itemize}

So without loss of generality $I''_k = I_k$.


    By Observation \ref{obs:pairstable}, $\tp(\a_{<k}/A)$ is generically stable.
    By the Strong Symmetry Lemma \ref{lem:strongsymm}  (note
    that $I_k\ind_AA$ since $A$
    is an extension base) we have
    $I_{k} \ind_A \a_{<k}$.

    By Fact \ref{fct:nonforkexist} again there is $I''_k
    \equiv_{A\a_{<k}} I_k$ such that $I''_k \ind_A I_{<k}$. Applying
    an automorphism over $A\a_{<k}$ like before, we may assume that
    $I''_k = I_k$ (this time the sequences $I_{<k}$ might change, but they still
    have all the desired properties).
    Note that it is still the case that $\a_{<k}
    \ind_A I_k$, hence by Observation \ref{obs:indispreserve} the
    sequence $I_k$ is indiscernible over $A\a_{<k}$.



    Since we could make $I_k$ as long as we wish to begin with, by
    Fact \ref{fct:Erdos Rado} there is $I''_k$ indiscernible over
    $AI_{<k}$ such that every $n$-type of $I''_k$ over $AI_{<k}$
    ``appears'' in $I_k$. In particular, we have  $I''_k \ind_A I_{<k}$ and (since $I_k$ is
    $A\a_{<k}$-indiscernible) $I''_k \equiv_{A\a_{<k}} I_k$.

    Let $\sigma \in \Aut(\fC/A\a_{<k})$ taking $I''_k$ onto (an
    initial segment of) $I_k$. Denote $I'_k = \sigma(I''_k)$,
    $I'_{<k} = \sigma(I_{<k})$. Clearly we still have
    \begin{enumerate}
    \item
        $\lseq{I'}{i}{k}$ are
        as in \ref{equ:mutindisc1}
    \item
        $I'_k$ is $AI'_{<k}$-indiscernible
    \item
        $I'_k \ind_A I'_{<k}$
    \item
        $I'_k$ starts with $\a_k$
    \end{enumerate}

    For $i<k$ let $B_i = A \cup\bigcup\set{I_j\colon j<k,j\neq i}$.
    By (iii) above, for every $i<k$ we have
    $I'_k \ind_{B_i}I'_i$. By (i) above $I'_i$ is
    $B_i$-indiscernible, hence by Observation \ref{obs:indispreserve}
    $I'_i$ is indiscernible over $AI'_{\neq i}$.

    Combining this with (ii) and (iii) above, we see that \lseq{I'}{i}{k+1}
    are as in \ref{equ:mutindisc1}, which completes the induction step.



\end{prf}

\begin{thm}\label{thm:dpweight}
    Let $A$ be an extension base (that is, no
    type over $A$ forks over $A$; e.g. $A$ is a model).
    Then for every type $p \in \tS(A)$, $\dpr(p) \ge \spwt(p)$
\end{thm}
\begin{prf}
    Let $\a\models p$, $\lseq{q}{i}{\al}$ generically stable, $\lseq{\b}{i}{\al}$
    ($\b_i \models q_i$) exemplify $\spwt(p) \ge \al$.
    Since $\tp(\a/A\b_i)$ divides over $A$, this is exemplified by
    an $A$-indiscernible \om-sequence $I_i$ starting with $\b_i$. By
    Lemma \ref{lem:mutindisc}, without loss of generality the
    sequence $I_i$ is indiscernible over $AI_{\neq i}$ for all $i$.

    Suppose $\tp(\a/A\b_i)$ $k_i$-divides over $A$, and $k_i$ is minimal such for $p$;
    that is, there
    is a formula $\ph_i(\x,\y)$ such that $\ph_i(\a,\b_i)$ holds,
    and the set
    $\set{\ph(\x,\b)\colon\b\in I_i}$
    is $k_i-1$-consistent with $p$, but $k_i$-inconsistent. Clearly $k_i \ge 2$. Let
    $$\psi_i(\x,\y_0,\ldots,\y_{k_i-2}) =
    \bigwedge_{\ell<k_i-1}\ph_i(\x,\y_\ell)$$
    and let $J_i$ be the sequence of $k_i-1$-tuples of elements
    of $I_i$ (``chunks'' of size $k_i-1$ from $I_i$), formally:
    $$J_i =
    \seq{\b_{i,\ell}\b_{i,\ell+1}\ldots\b_{i,\ell+k_i-2}\colon\ell<\om}$$

    Denote the first element of $J_i$ by $\c_i$.
    Clearly
    \begin{itemize}
    \item
        $J_i$ is indiscernible over $AJ_{\neq i}$
    \item
         $\psi(\a,\c_i)$ holds
    \item
        The set $\set{\psi(\x,\c)\colon c\in J_i}$ is 2-inconsistent with $p$
    \end{itemize}
    Now the sequences \lseq{J}{i}{\al} form an array which together
    with formulas $\psi_i$ give a randomness pattern of depth $\al$
    for $p$, as required.


\end{prf}

\begin{cor}\label{cor:finiteweight}
\begin{enumerate}
\item
    In a strongly dependent theory every type has finite stable weight.
\item
    A strongly dependent type has finite stable weight.
\item
    A stable theory is strongly stable if and only if every type has finite weight.
\end{enumerate}
\end{cor}

Corollary \ref{cor:finiteweight}(iii) was observed independently by
Hans Adler in \cite{Ad2}.

Further properties of stable weight will be investigated elsewhere.
On the different notions of weight etc in dependent theories see
also works by Adler \cite{Ad2}, Alf Onshuus and the author
\cite{OnUs1}, \cite{OnUs2}. In section 4 of \cite{OnUs2} different
attempts are made in order to remove the assumption of generic
stability and prove an analogue of Theorem \ref{thm:dpweight} for
weight and not stable weight, which leads to several general results
on mutual indiscernibility (some generalize Lemma
\ref{lem:mutindisc}) and the behavior of forking in dependent and
strongly dependent theories, but the main goal has not yet been
achieved.



\appendix

\section{Ultrafilters and measures}  
The following definitions are motivated by \cite{Sh715}, \cite{Ke}
and \cite{HPP}.
Our hope is that they might clarify the connections between
ultrafilters used by Shelah in \cite{Sh715}, coheirs and Shelah
sequences, Definition \ref{dfn:shseq} (see Discussion
\ref{dsc:types-filters}).

In order to make the appendix more self-contained, we will repeat
here certain definitions and remarks from Section 3. Note that it is
convenient and natural to define these notions in this generality.

\begin{dfn}
  Let $A,B,C$ sets.
  \begin{enumerate}
    \item
      Recall that we denote the boolean algebra of $C$-definable subsets of $A^m$ by
      $\Def_m(A,C)$,
      $\Def(A,C) = \bigcup_{m<\om}\Def_m(A,C)$.
    \item
      A $(C,m)$-\emph{Keisler measure} on $A$, or an
      $m$-measure \emph{on $A$ over $C$} is a finitely additive probability
      measure on $\Def_m(A,C)$. We omit $C$ if $C=A$. We omit $m$ if it is clear
      from the context.
    \item
      We will call a $(C,m)$-Keisler measure on $\fC$ a \emph{full} $m$-Keisler measure
      over $C$. A \emph{global} $m$-measure is a full measure over
      $\fC$, i.e. a measure on $\Def_m(\fC)$.
    \item
      A Keisler measure $\mu$ on $\Def_m(B,C)$ is \emph{supported on} $A\subseteq B$
      if there exists a measure $\mu_0$ on $\Def_m(A,C)$ such that
      for every formula $\ph(\x,\c)$ over $C$ with $\len(\x) = m$ we have
      $\mu(\ph^\fC(\x,\c)) = \mu_0(\ph^\fC(\x,\c)\cap A^m)$. We denote $\mu_0$ (if exists)
      by $\mu \downarrow A$.
  \end{enumerate}
\end{dfn}

\begin{rmk}
\begin{enumerate}
\item
  Note that a full $\set{0,1}$-measure over $C$ (i.e. an ultrafilter \FU on $\Def(\fC,C)$) precisely
  corresponds to
  a complete type over $C$.
\item
  Note that a type $p \in \tS(C)$ is finitely satisfiable in $A$ iff the
  appropriate ultrafilter (measure) \FU on $\Def(\fC,C)$ is supported on A. In this case
  the measure $\fU\downarrow A$ (see the definition of ``supported on $A$'' above)
    defines over $C$ the same type as $\fU$, that is,$\Av(\fU\downarrow A,C) =
    p$.
\end{enumerate}
\end{rmk}

\begin{dfn}
\label{dfn:Aaverage}
\begin{enumerate}
\item
  Let $A \subseteq B$, and let $\mu$ be a Keisler measure on $A^m$ over $C$,
  i.e. $\mu$ a measure on $\Def_m(A,C)$.
  We define
  the \emph{lifted measure} of $\mu$ over $B$ (denoted $\mu\uparrow B$)
  on $\Def_m(B,C)$ by
  $\mu\uparrow B(\ph(\x,\b)^\fC \cap B^m)  = \mu(\ph^\fC(\x,\b)\cap A^m)$.

  In particular, for $B = \fC$, call $\mu \uparrow \FC$ the \emph{full
    lifted measure} of $\mu$.
\item
  Let $A,C$ sets, $\fU$ an ultrafilter on $\Def_m(A,C)$.
  Recall that we define
  the \emph{average type} of $\fU$ over $C$) by
  \[
     \Av(\fU,C) = \big\{\ph(\x,\c) \colon \c \in C \; \mbox{and} \;\{\a \in A^m \colon \ph(\a,\c)\} \in \fU
     \big\}
  \]
\end{enumerate}
\end{dfn}

\begin{obs}
  For $A,B,C,\mu$ as in (i) above, $\mu\uparrow B$ is a Keisler measure on $\Def_m(B,C)$ supported on $A$,
  $(\mu \uparrow B) \downarrow A = \mu$. Moreover, for $A \subseteq A' \subseteq B$,
  $(\mu \uparrow B) \downarrow A' = \mu \uparrow A'$.
\end{obs}

\begin{rmk}
  Note that in clause (ii) of Definition \ref{dfn:Aaverage} we obtain a complete type over $C$,
  therefore a full measure over $C$, i.e. an ultrafilter $\frak V$
  on $\Def(\fC,C)$.
  Clearly, this ultrafilter corresponds to the full lifted measure of $\FU$ over $C$, that is,
  $\frak V = \fU\uparrow \fC$. So the second definition is a particular case of the first one.
\end{rmk}


One can also generalize the definition of nonsplitting to measures:

\begin{dfn}
    A Keisler measure $\mu$ on $\Def(B,D)$ \emph{does not
    split} over $A$ if $\mu(\ph(\x,\b))$ depends only
    on $\tp(\b/A)$ for every formula $\ph(\x,\y)$ and $\b \in B$.
\end{dfn}

So a global measure doesn't split over a set $A$ if it is invariant
under the action of the automorphism group of \FC over $A$.



Many properties of nonsplitting types remain true when passing to
measures. We will just make a few small observations.

A notion of a definable measure (which is the analogue of a
definable type) was studied in \cite{HPP}. The following is the
analogue of Observation \ref{obs:nonsplit}:

\begin{obs}
\label{obs:Anonsplit}
\begin{enumerate}
\item
  If a Keisler measure over $B$ is supported on $A \subseteq B$,
  then it does not split over $A$.
\item
  If a Keisler measure over $B$ is definable over $A \subseteq B$,
  then it does not split over $A$.
\end{enumerate}
\end{obs}

Given a set $A$, there are boundedly many measures which do not
split over $A$:

\begin{obs}\label{obs:Afewnonsplit}
    Let $A$ be a set. Then there are at most $\beth_2(|A|+|T|) = {2^{2^{|A|+|T|}}}$
    Keisler measures $\mu$
    over $\FC$ which do not split over $A$.
\end{obs}
\begin{prf}
    For each type $r(\y)$ over $A$ and each formula $\ph(\x,\y)$, we
    have (at most) continuum many options for $\mu(\ph(\x,\c))$
    where $\c\models r$. This determines $\mu$ completely. So we
    have $2^{{\aleph_0}^{(2^{|A|+|T|})}} = 2^{2^{|A|+|T|}}$ nonsplitting
    measures.
\end{prf}

\begin{dsc}\label{dsc:types-filters}
Note that in \cite{Sh715} Shelah works with types of the form $p \in
\tS(C)$ finitely satisfiable in some $A$, which corresponds to the
following situation: an ultrafilter $\fU$ on $\Def(\fC,C)$ which is
supported on $A$, i.e. comes from a measure on $\Def(A,C)$. (Shelah
works with ultrafilters on the algebra of all subsets of $A$, but of
course it is enough to restrict oneself to the definable subsets).
\end{dsc}


\bibliography{common}

\begin{bibdiv}
\begin{biblist}

\bib{Ad}{article}{
      author={Adler, Hans},
       title={An introduction to theories without the independence property},
     journal={submitted},
}

\bib{Ad2}{article}{
      author={Adler, Hans},
       title={Strong theories, burden, and weight},
     journal={in preparation},
}

\bib{Dol}{article}{
      author={Dolich, Alf},
       title={Forking and independence in o-minimal theories},
        date={2004},
     journal={J. Symbolic Logic},
      volume={69},
       pages={215\ndash 240},
}

\bib{HHM}{article}{
      author={Haskell, Deirdre},
      author={Hrushovski, Ehud},
      author={Macpherson, Dugald},
       title={Stable domination and independence in algebraically closed valued
  fields},
     journal={submitted},
}

\bib{HaOn2}{article}{
      author={Hasson, Assaf},
      author={Onshuus, Alf},
       title={The geometry of minimal types in dependent theories},
     journal={in preparation},
}

\bib{HaOn1}{article}{
      author={Hasson, Assaf},
      author={Onshuus, Alf},
       title={Unstable structures definable in o-minimal theories},
     journal={submitted},
}

\bib{Hr}{article}{
      author={Hrushovski, Ehud},
       title={Valued fields, metastable groups},
     journal={in preparation},
}

\bib{HPP}{article}{
      author={Hrushovski, Ehud},
      author={Peterzil, Yaacov},
      author={Pillay, Anand},
       title={Groups, measure and the nip},
     journal={Journal AMS},
        note={to appear},
}

\bib{HP}{article}{
      author={Hrushovski, Ehud},
      author={Pillay, Anand},
       title={Nip and invariant measures},
     journal={preprint},
}

\bib{Ke}{article}{
      author={Keisler, H.~Jerome},
       title={Measures and forking},
        date={1987},
     journal={Ann. Pure Appl. Logic},
      volume={34},
      number={2},
       pages={119\ndash 169},
}

\bib{LP}{article}{
      author={Lascar, Daniel},
      author={Poizat, Bruno},
       title={An introduction to forking},
        date={1979},
     journal={J. Symbolic Logic},
      volume={44},
      number={3},
       pages={330\ndash 350},
}

\bib{OnPe}{article}{
      author={Onshuus, Alf},
      author={Peterzil, Yaacov},
       title={A note on stable sets, groups, and theories with nip},
        date={2007},
     journal={Mathematical Logic Quarterly},
      volume={53},
}

\bib{OnUs1}{article}{
      author={Onshuus, Alf},
      author={Usvyatsov, Alexander},
       title={Dp-minimality, strong dependence and weight},
     journal={in preparation},
}

\bib{OnUs2}{article}{
      author={Onshuus, Alf},
      author={Usvyatsov, Alexander},
       title={Orthogonality and domination in unstable theories},
     journal={preprint},
}

\bib{Sh783}{article}{
      author={Shelah, Saharon},
       title={Dependent first order theories, continued},
     journal={Israel Journal of Mathematics},
        note={accepted},
}

\bib{Sh863}{article}{
      author={Shelah, Saharon},
       title={Strongly dependent theories},
     journal={submitted},
}

\bib{Sh715}{article}{
      author={Shelah, Saharon},
       title={Classification theory for elementary classes with the dependence
  property---a modest beginning},
        date={2004},
     journal={Sci. Math. Jpn.},
      volume={59},
      number={2},
       pages={265\ndash 316},
        note={Special issue on set theory and algebraic model theory},
}

\bib{Wa}{book}{
      author={Wagner, Frank~O.},
       title={Simple theories},
      series={Mathematics and Its Applications},
   publisher={Kluwer Academic Publishers},
        date={2000},
      volume={503},
}

\end{biblist}
\end{bibdiv}
\bibliographystyle{alpha}

\end{document}